\documentclass[11pt]{article} 

\usepackage{fullpage}
\usepackage{amsthm}
\usepackage{amssymb}
\usepackage{amsmath}
\usepackage[mathscr]{eucal}
\usepackage{graphicx}
\usepackage{centernot}
\usepackage{textcomp}
\usepackage{afterpage}
\usepackage{esint}
\usepackage{xspace}
\usepackage{siunitx}
\usepackage{tikz}
\usepackage{hyperref}
\usepackage{bm}
\usepackage{wasysym}
\usepackage{accents}
\usepackage{xcolor}

\usepackage[T1]{fontenc}
\usepackage[utf8]{inputenc}
\usepackage{authblk}

\newcommand{\abs}[1]{\left| #1 \right|}

\newcommand{\C}{\mathbb{C}}

\newcommand{\N}{\mathbb{N}}

\newcommand{\norm}[1]{\|{#1}\|}

\newcommand{\p}{\partial}

\newcommand{\R}{\mathbb{R}}

\newcommand{\set}[2]{\left\{{#1}:{#2}\right\}}

\newcommand{\trace}[1]{\textrm{tr}\left(#1\right)}

\theoremstyle{plain}
\newtheorem{theorem}[subsection]{Theorem}
\newtheorem{proposition}[subsection]{Proposition}
\newtheorem{corollary}[subsection]{Corollary}
\newtheorem{lemma}[subsection]{Lemma}

\theoremstyle{definition}
\newtheorem{definition}[subsection]{Definition}
\newtheorem{example}[subsection]{Example}

\theoremstyle{remark}
\newtheorem{remark}[subsection]{Remark}

\theoremstyle{conjecture}

\theoremstyle{exercise}

\numberwithin{equation}{section}

\title{Propagation of Global Analytic Singularities for Schr\"{o}dinger Equations with Quadratic Hamiltonians}
\author{Francis White}
\affil{University of California Los Angeles}
\date{}

\begin{document}

\maketitle

\begin{abstract}
We study the propagation in time of $1/2$-Gelfand-Shilov singularities, i.e. global analytic singularities, of tempered distributional solutions of the initial value problem
\begin{align*}
\begin{cases}
u_t + q^w(x,D) u = 0 \\
u|_{t=0} = u_0,
\end{cases}
\end{align*}
on $\R^n$, where $u_0$ is a tempered distribution on $\R^n$, $q=q(x,\xi)$ is a complex-valued quadratic form on $\R^{2n} = \R^n_x \times \R^n_\xi$ with nonnegative real part $\textrm{Re} \ q \ge 0$, and $q^w(x,D)$ is the Weyl quantization of $q$. We prove that the $1/2$-Gelfand-Shilov singularities of the initial data that are contained within a distinguished linear subspace of the phase space $\R^{2n}$, called the \emph{singular space of $q$}, are transported by the Hamilton flow of $\textrm{Im} \ q$, while all other $1/2$-Gelfand-Shilov singularities are instantaneously regularized. Our result extends the observation of Hitrik, Pravda-Starov, and Viola '18 that this evolution is instantaneously globally analytically regularizing when the singular space of $q$ is trivial.
\end{abstract}

\tableofcontents

\section{Introduction and Statement of Results}

In this paper we consider the initial value problem for the Schr\"{o}dinger equation
\begin{align} \label{non self adjoint Schrodinger equation}
    \begin{cases}
    \p_t u + q^w(x,D)u = 0, \ \ t \ge 0, \ x \in \R^n, \\
    u|_{t=0}= u_0,
    \end{cases}
\end{align}
where the initial data $u_0$ is a tempered distribution on $\R^n$, $q = q(x,\xi)$ is a complex-valued quadratic form defined on the phase space $T^*\R^n \cong \R^n_x \times \R^n_\xi$ with non-negative real part $\textrm{Re} \ q \ge 0$ , and $q^w(x,D)$ is the Weyl quantization of $q$ given by
\begin{align*}
    q^w(x,D)u(x) = \frac{1}{(2\pi)^n} \int \int e^{i(x-y) \cdot \xi} q \left(\frac{x+y}{2}, \xi \right) u(y) \, dy \, d\xi, \ \ \ u \in \mathcal{S}'(\R^n),
\end{align*}
in the sense of distributions. This class of equations comprises a number of well-known examples, including the free Schr\"{o}dinger equation where $q(x,\xi) = i\abs{\xi}^2$, the harmonic oscillator where $q(x,\xi) = i(\abs{x}^2+\abs{\xi}^2)$, the heat equation where $q(x,\xi) = \abs{\xi}^2$, as well as the Kramers-Fokker-Planck equation with quadratic potential where $q(x,v,\xi,\eta) = \eta^2+1/4v^2+i(v \cdot \xi-ax \cdot \eta)$ with $(x,v,\xi,\eta) \in \R^{4n} = \R^{2n}_{x,v} \times \R^{2n}_{\xi, \eta}$ and $a \in \R \backslash \{0\}$ a constant.

The problem has been studied by a number of authors, and the majority of works in this area (\cite{GaborSingularities}, \cite{PolynomialSingularities}, \cite{time_dependent}) have focused on the propagation of Gabor singularities. The Gabor singularities of a tempered distribution $u \in \mathcal{S}'(\R^n)$ are captured by the Gabor wavefront set of $u$, which may be defined as follows: given $u \in \mathcal{S}'(\R^n)$, the \emph{Gabor wavefront of $u$}, denoted $\textrm{WF}_G(u)$, is the complement in $\R^{2n} \backslash \{(0,0)\}$ of all points $(x_0,\xi_0)$ for which there exists a symbol $a \in C^\infty(\R^{2n})$ satisfying
\begin{enumerate}
    \item for all $\alpha, \beta \in \N^n$ there exists $C>0$ such that
    \begin{align*}
        \abs{\p^\alpha_x \p^\beta_\xi a(x,\xi)} \le C \langle(x,\xi)\rangle^{-\abs{\alpha}-\abs{\beta}}, \ \ \  (x,\xi) \in \R^{2n},
    \end{align*}
    \item there exists an open conic neighborhood $V$ of $(x_0, \xi_0)$ in $\R^{2n} \backslash \{(0,0)\}$ and $c>0$ so that $\abs{a(x,\xi)} \ge c$ within $V$ for all $\abs{(x,\xi)}$ large, and
    \item $a^w(x,D) u \in \mathcal{S}(\R^n)$.
\end{enumerate}
In particular, a tempered distribution is a Schwartz function if and only if its Gabor wavefront set is empty. Thus, the Gabor wavefront set of a tempered distribution measures its deviation from Schwartz regularity in the sense of both smoothness and decay as $\abs{x} \rightarrow \infty$. For more information, see \cite{GaborWavefront} or the recent survey article \cite{Intro2Gabor}.

The propagation of exponential phase space singularities by the evolution (\ref{non self adjoint Schrodinger equation}) has also been studied. In the work \cite{GelfandShilov}, the authors investigated the propagation of $s$-Gelfand-Shilov singularities for $s>1/2$ with initial data belonging to classes of distributions larger than $\mathcal{S}'(\R^n)$. For $s \ge 1/2$, the $s$-Gelfand-Shilov wavefront of $u \in \mathcal{S}'(\R^n)$ is a closed conic subset of $\R^{2n} \backslash \{(0,0)\}$ that may be conveniently defined using metaplectic Fourier-Bros-Iagolnitzer (FBI) transforms. We recall (see \cite{Minicourse} or Chapter 13 of \cite{SemiclassicalAnalysis}) that a metaplectic FBI transform on $\R^n$ is a Fourier integral operator $\mathcal{T}_\varphi: \mathcal{S}'(\R^n) \rightarrow \textrm{Hol}(\C^n)$ of the form
\begin{align}
	\mathcal{T}_\varphi u(z) = c_\varphi \int_{\R^n} e^{i \varphi(z,y)} u(y) \, dy, \ \ u \in \mathcal{S}'(\R^n),
\end{align}
where $c_\varphi \neq 0$ is a normalizing constant and $\varphi(z,y)$ is a holomorphic quadratic form on $\C^{2n} = \C^n_z \times \C^n_y$ such that $\det{\p^2_{zy} \varphi} \neq 0$ and $\textrm{Im} \ \p^2_{yy} \varphi >0$. The phase function $\varphi$ generates a complex linear canonical transformation $\kappa_\varphi: \C^{2n} \rightarrow \C^{2n}$ given implicitly by
\begin{align}
	\kappa_\varphi: (y, -\p_y \varphi(z,y)) \mapsto (z, \p_z \varphi(z,y)), \ \ (z,y) \in \C^{2n},
\end{align}
and we have
\begin{align}
	\kappa_\varphi(\R^{2n}) = \set{\left(z, \frac{2}{i} \p_z \Phi(z) \right)}{z \in \C^n},
\end{align}
where
\begin{align}
\Phi(z) = \sup_{y\in \R^n}(-\textrm{Im} \ \varphi(z,y))	
\end{align}
is the strictly plurisubharmonic weight associated to $\varphi$. If $\varphi$ is a metaplectic FBI phase function and $\kappa_\varphi: \C^{2n} \rightarrow \C^{2n}$ is the complex linear canonical transformation generated by $\varphi$, let 
\begin{align} \label{def_kappa_flat_intro}
\kappa^\flat_\varphi = \pi_1 \circ \kappa_\varphi: \C^{2n} \rightarrow \C^n,
\end{align}
where $\pi_1: (z,\zeta) \mapsto z$ is the projection in $\C^{2n}$ onto the first factor. We say that a point $(x_0, \xi_0) \in \R^{2n} \backslash \{(0,0)\}$ does not belong to the \emph{$s$-Gelfand-Shilov wavefront set of $u$}, denoted $\textrm{WF}^s(u)$, if there exists a metaplectic FBI phase function $\varphi=\varphi(z,y)$, $(z,y) \in \C^n \times \C^n$, with associated plurisubharmonic weight $\Phi$, and complex canonical transformation $\kappa_\varphi: \C^{2n} \rightarrow \C^{2n}$ such that
\begin{align*}
    \exists C, c>0: \ \abs{\mathcal{T}_\varphi u(z)} \le C e^{\Phi(z)-c \abs{z}^{1/s}}
\end{align*}
for all $z$ in some open conic neighborhood of $\kappa^\flat_\varphi(x_0, \xi_0)$ in $\C^n \backslash \{0\}$. For a review of the basic properties of metaplectic FBI transforms on $\R^n$, see Section 2 below. The notion of the $s$-Gelfand-Shilov wavefront of a tempered distribution was first introduced in the case $s=1/2$ by L. H\"{o}rmander in \cite{HyperbolicOperators} under the name `analytic wavefront set.' Roughly speaking, if $u \in \mathcal{S}'(\R^n)$, then $\textrm{WF}^{1/2}(u)$ measures the failure of $u$ to admit an extension to a holomorphic function $U$ on $\C^n$ satisfying
\begin{align} \label{Gelfand-Shilov estimate first reference}
    \abs{U(w)} \le C e^{-c \abs{x}^2+C \abs{y}^2}, \ \ \ \forall w = x+iy \in \C^n,
\end{align}
for some $C,c>0$. A function $u$ on $\R^n$ admitting an extension $U \in \textrm{Hol}(\C^n)$ obeying the estimate (\ref{Gelfand-Shilov estimate first reference}) is called a \emph{Gelfand-Shilov test function}. Thus, $\textrm{WF}^{1/2}(u)$ is a microlocal measure of how a distribution $u \in \mathcal{S}'(\R^n)$ fails to be a Gelfand-Shilov test function. For additional information concerning general $s$-Gelfand-Shilov wavefront sets and classes of Gelfand-Shilov ultradistributions, see \cite{GelfandShilov}. In the present work, our objective is to study the propagation in time of $s$-Gelfand-Shilov singularities by the evolution (\ref{non self adjoint Schrodinger equation}) for the limiting value $s=1/2$. To the best of our knowledge, this case has not yet been explored.

We now proceed to state the main results of this paper. Let $\R^{2n}$ be equipped with the standard symplectic form
\begin{align} \label{Hamilton matrix of q}
    \sigma((x,\xi), (y,\eta)) = \xi \cdot y - x \cdot \eta, \ \ \ (x,\xi), \ (y,\eta) \in \R^{2n}.
\end{align}
Suppose $q: \R^{2n} \rightarrow \C$ is a complex-valued quadratic form and let $q(\cdot, \cdot)$ denote its symmetric $\C$-bilinear polarization. Because $\sigma$ is non-degenerate, there is a unique $F \in M_{2n \times 2n}(\C)$ such that
\begin{align*}
    q((x,\xi), (y,\eta)) = \sigma((x,\xi), F(y,\eta))
\end{align*}
for all $(x,\xi), (y,\eta) \in \R^{2n}$. This matrix $F$ is called the \emph{Hamilton map} or \emph{Hamilton matrix of $q$} (see Section 21.5 of \cite{HormanderIII}). Explicitly, the Hamilton matrix of $q$ is given by
\begin{align*}
    F = J Q
\end{align*}
where
\begin{align*}
    J = \begin{pmatrix} 0 & I \\ -I & 0 \end{pmatrix}
\end{align*}
is the standard $2n \times 2n$ symplectic matrix and $Q \in M_{2n \times 2n}(\C)$ is the unique complex symmetric matrix such that
\begin{align*}
    q(X) = Q X \cdot X, \ \ \ X \in \R^{2n}.
\end{align*}
Let
\begin{align*}
    \textrm{Re} \ F = \frac{F+\overline{F}}{2} \ \ \textrm{and} \ \ \textrm{Im} \ F = \frac{F-\overline{F}}{2i}
\end{align*}
be the real and imaginary parts of $F$ respectively. The \emph{singular space $S$ of $q$} is defined as the following finite intersection of kernels:
\begin{align} \label{definition of singular space}
    S = \left(\bigcap_{j=0}^{2n-1} \ker{\left[(\textrm{Re} \ F)(\textrm{Im} \ F)^j\right]} \right) \cap \R^{2n}.
\end{align}
The singular space was first introduced by M. Hitrik and K. Pravda-Starov in \cite{QuadraticOperators} where it arose naturally in the study of spectra and semi-group smoothing properties for non-self adjoint quadratic differential operators. The concept of the singular space has since been shown to play a key role in the understanding of hypoelliptic and spectral properties of non-elliptic quadratic differential operators. See for instance \cite{SemiclassicalHypoelliptic}, \cite{EigenvaluesAndSubelliptic}, \cite{ContractionSemigroup}, \cite{SubellipticEstimatesQuadraticDifferentialOperators}, \cite{NonEllipticQuadraticFormsandSemiclassicalEstimates}, and \cite{SpectralProjectionsAndResolventBounds}.

In the work \cite{GeneralizedMehler}, it was shown that when the quadratic form $q$ has non-negative real part $\textrm{Re} \ q \ge 0$, the maximal closed realization of the quadratic differential operator $q^w(x,D)$ on $L^2(\R^n)$ is maximally accretive and generates a contraction semigroup $(e^{-tq^w(x,D)})_{t \ge 0}$. In this same work it was also proven that for each $t \ge 0$ the operator $e^{-tq^w(x,D)}$ both restricts to a continuous linear transformation
\begin{align} \label{bounded Schwartz to Schwartz}
    \mathcal{S}(\R^n) \rightarrow \mathcal{S}(\R^n)
\end{align}
and admits a unique extension to a continuous linear transformation
\begin{align} \label{bounded S prime to S prime}
    \mathcal{S}'(\R^n) \rightarrow \mathcal{S}'(\R^n).
\end{align}
Consequently,
\begin{align*}
    e^{-tq^w(x,D)} u_0
\end{align*}
is a well-defined element of $\mathcal{S}'(\R^n)$ for any $u_0 \in \mathcal{S}'(\R^n)$ and $t \ge 0$. Our goal is to understand the relationship between $\textrm{WF}^{1/2}(e^{-tq^w(x,D)}u_0)$ and $\textrm{WF}^{1/2}(u_0)$ in terms of the Hamiltonian dynamics generated by $q$.

Let us point out that the topic of the partial Gelfand-Shilov regularizing properties for the semigroup $e^{-tq^w(x,D)}$ has received considerable attention over the last several years. The general aim of work in this direction has been to prove sharp microlocal Gelfand-Shilov smoothing estimates in directions transverse to the singular space $S$ for various time regimes. We mention the works \cite{SubellipticEstimates}, \cite{PartialGS}, and \cite{alphonse2020polar}. In particular, it follows from Theorem 1.2 in \cite{SubellipticEstimates} that
\begin{align*}
    \textrm{WF}^{1/2} \left(e^{-tq^w(x,D)} u_0 \right) = \emptyset
\end{align*}
for all $u_0 \in L^2(\R^n)$ and $0<t \ll 1$ when $S = \{0\}$. In other words, if the singular space of the quadratic form $q$ is trivial, then the propagator for the evolution (\ref{non self adjoint Schrodinger equation}) is `instantaneously globally analytically regularizing.' Our main theorem extends this observation to the case of general $S$.

\begin{theorem} \label{main theorem of the paper}
Let $q$ be a complex-valued quadratic form on $\R^{2n}$ with $\textrm{Re} \ q \ge 0$ and $S$ be the singular space of $q$. For every $u_0 \in \mathcal{S}'(\R^n)$ and $t >0$, we have
\begin{align} \label{main inclusion}
\textrm{WF}^{1/2}(e^{-tq^w(x,D)} u_0) = \exp{(t H_{\textrm{\emph{Im}} \ q})(\textrm{WF}^{1/2}(u_0) \cap S)}.
\end{align}
Here $H_{\textrm{Im} \ q}$ denotes the Hamilton vector field of $\textrm{\emph{Im}} \ q$ on $\R^{2n}$ taken with respect to the symplectic form $\sigma$, and $\exp{\left(t H_{\textrm{\emph{Im}} \ q}\right)}$ is the Hamilton flow of $\textrm{\emph{Im}} \ q$ at time $t$.
\end{theorem}

\begin{remark} 
In \cite{GaborSingularities} it was shown that
\begin{align} \label{Gabor inclusion}
    \textrm{WF}_G(e^{-tq^w(x,D)} u_0) \subset \textrm{exp}(tH_{\textrm{Im} \ q})(\textrm{WF}_G(u_0) \cap S), \ \ \ \forall u_0 \in \mathcal{S}'(\R^n), \ \forall t>0,
\end{align}
and that equality need not hold for any $t>0$. In \cite{ExponentialSingularities}, it was established that for $s>1/2$, there is the inclusion
\begin{align} \label{s-Gelfand-Shilov inclusion}
    \textrm{WF}^{s}(e^{-tq^w(x,D)} u_0) \subset \exp{(t H_{\textrm{Im} \ q})(\textrm{WF}^{s}(u_0) \cap S)} \ \ \ \forall u_0 \in \mathcal{S}'(\R^n), \ \forall t>0.
\end{align}
In contrast to these results, Theorem \ref{main theorem of the paper} shows that the propagator $e^{-tq^w(x,D)}$ never regularizes any of the $1/2$-Gelfand-Shilov singularities of $u_0$ that lie within $S$. Thus it is possible to recover $\textrm{WF}^{1/2}(u_0) \cap S$ from knowledge of $\textrm{WF}^{1/2}(e^{-tq^w(x,D)}u_0)$ for some $t>0$. So far as we are aware, it is not known if the inclusion in (\ref{s-Gelfand-Shilov inclusion}) is in general strict for $s>1/2$. It would be of interest to determine the range of $s \ge 1/2$ for which equality holds in (\ref{s-Gelfand-Shilov inclusion}) for all $u_0 \in \mathcal{S}'(\R^n)$ and $t>0$.
\end{remark}

Our approach to the proof of Theorem \ref{main theorem of the paper} is based entirely on FBI transform techniques. Let $\varphi(z,y)$, $(z,y) \in \C^{2n}$, be a metaplectic FBI phase function with associated plurisubharmonic weight $\Phi(z)$, $z \in \C^n$, and complex linear canonical transformation $\kappa_\varphi: \C^{2n} \rightarrow \C^{2n}$. Our strategy is to study the conjugated propagator
\begin{align} \label{conjugated propagator in intro}
    \mathcal{T}_\varphi \circ e^{-tq^w(x,D)} \circ \mathcal{T}_\varphi^*, \ \ \ t \ge 0,
\end{align}
acting on the extended Bargmann space $H^{-\infty}_\Phi(\C^n)$, which is the image of $\mathcal{S}'(\R^n)$ by $\mathcal{T}_\varphi$. The extended Bargmann space $H^{-\infty}_\Phi(\C^n)$ is the space of all $u \in \textrm{Hol}(\C^n)$ such that
\begin{align*}
    \int_{\C^n} \langle z \rangle^{2s} \abs{u(z)}^2 e^{-2\Phi(z)} \, L(dz) < \infty,
\end{align*}
for some $s \in \R$. See Section 2 below for a full discussion of the space $H^{-\infty}_\Phi(\C^n)$ and related exponentially weighted spaces of entire functions. By Egorov's theorem, the operator (\ref{conjugated propagator in intro}) is the evolution semigroup generated by the complex Weyl differential operator $\tilde{q}^w(z,D_z)$, where $\tilde{q} = q \circ \kappa_\varphi^{-1}$. We note that this approach differs significantly from that of the works \cite{GaborSingularities}, \cite{PolynomialSingularities}, \cite{ExponentialSingularities}, and \cite{time_dependent}, which rely heavily on representations of the Schwartz kernel of $e^{-tq^w(x,D)}$ as a Gaussian oscillatory integral in the sense of \cite{GeneralizedMehler}. In the present work, we use an elementary geometrical optics construction to show that the semigroup $e^{-t\tilde{q}^w(z,D_z)}$ is a Fourier integral operator in the complex domain in the sense of Sj\"{o}strand \cite{AnalyticMicrolocal_Analysis}. Motivated by recent results concerning Bergman representations of metaplectic Fourier integral operators in the work \cite{ComplexFIOs}, we establish an integral representation for the semigroup $e^{-t\tilde{q}^w(z,D)}$ on the FBI transform side of the form
\begin{align} \label{Bergman_form_in_intro}
	e^{-t\tilde{q}^w(z,D)} u(z) = \hat{a}(t) \int_{\C^n} e^{2 \Psi_t(z,\overline{w})} u(w) e^{-2\Phi(w)} \, L(dw), \ \ u \in H_\Phi(\C^n), \ \ t \ge 0,
\end{align}
where $\hat{a}(t) \in \C$ depends analytically on $t$ and $\Psi_t(\cdot, \cdot)$ is a holomorphic quadratic form on $\C^{2n}$ whose coefficients depend analytically on $t$. Here we denote the Lebesgue measure on $\C^n = \C^n_z$ by $L(dz)$ and $H_\Phi(\C^n)$ is the \emph{Bargmann space}
\begin{align*}
	H_\Phi(\C^n) = L^2(\C^n, e^{-2\Phi(z)} \, L(dz)) \cap \textrm{Hol}(\C^n),
\end{align*}
which is the unitary image of $L^2(\R^n)$ under the FBI transform $\mathcal{T}_\varphi$.
Writing (\ref{Bergman_form_in_intro}) as the contour integral
\begin{align}
	e^{-t\tilde{q}^w(z,D)} u(z) = a(t) \iint_{\Gamma} e^{2 \Psi_t(z,\theta)-2\Psi(w,\theta)} u(w) \, dw \wedge d\theta, \ \ u\in H_\Phi(\C^n),
\end{align}
where
\begin{align}
	\Gamma = \set{(w,\theta)}{\theta = \overline{w}}
\end{align}
is the anti-diagonal in $\C^{2n}$, $a(t) = (i/2)^n \hat{a}(t)$, and $\Psi(\cdot, \cdot)$ is the polarization of $\Phi$, i.e. the unique holomorphic quadratic form on $\C^{2n}$ such that $\Psi(z,\overline{z}) = \Phi(z)$ for all $z \in \C^n$, we find that $e^{-t\tilde{q}^w(z,D)}$ is a metaplectic Fourier integral operator in the complex domain whose phase function generates the graph of the Hamilton flow $\tilde{\kappa_t} = \exp{(tH_{-i\tilde{q}})}: \C^{2n} \rightarrow \C^{2n}$ of $-i\tilde{q}$ at time $t$ for each $t \ge 0$. In particular, for every $t \ge 0$, the flow $\tilde{\kappa}_t$ is positive relative to the maximally totally real subspace
\begin{align} \label{definition of I lagrangian}
    \Lambda_\Phi = \set{\left(z, \frac{2}{i}\p_z \Phi(z) \right)}{z \in \C^n}
\end{align}
of $\C^{2n}$ in the sense that
\begin{align*}
    \frac{1}{i} \left(\sigma(\tilde{\kappa}_t(X), \iota_{\Lambda_\Phi}(\tilde{\kappa}_t(X))) - \sigma(X, \iota_{\Lambda_\Phi} X)\right) \ge 0, \ \ \ X \in \C^{2n},
\end{align*}
where $\iota_{\Lambda_{\Phi}}$ is the unique antilinear involution of $\C^{2n}$ fixing $\Lambda_\Phi$. By using some recent results from \cite{ComplexFIOs} concerning complex canonical transformations of $\C^{2n}$ that are positive relative to a pair of maximally totally real subspaces of the form {(\ref{definition of I lagrangian}) for some strictly plurisubharmonic quadratic form $\Phi$ on $\C^{2n}$}, we show that it is possible to find $\hat{a}(t)$ and $\Psi_t(\cdot, \cdot)$ so that (\ref{Bergman_form_in_intro}) is valid for all non-negative times $t \ge 0$. In particular, the Bergman representation (\ref{Bergman_form_in_intro}) provides an alternative to H\"{o}rmander's generalized Mehler formula \cite{GeneralizedMehler} for representing the semigroup $e^{-tq^w(x,D)}$. We feel that this result may be of independent interest, and we plan to explore additional applications in the future.

Having established the representation (\ref{Bergman_form_in_intro}), we can invoke some generalities concerning how metaplectic Fourier integral operators associated to complex canonical transformations of $\C^{2n}$ propagate and/or regularize $1/2$-Gelfand-Shilov singularities. Given $u \in H^{-\infty}_\Phi(\C^n)$, we define the \emph{$1/2$-Gelfand-Shilov wavefront set of $u$ relative to $\Phi$} as the complement in $\C^n \backslash \{0\}$ of all points $z_0$ for which there exist $C,c>0$ such that
\begin{align*}
    \abs{u(z)} \le C e^{\Phi(z)-c\abs{z}^2}
\end{align*}
for all $z$ in some open conic neighborhood of $z_0$ in $\C^n \backslash \{0\}$. Thus, if $u = \mathcal{T}_\varphi v$ where $v \in \mathcal{S}'(\R^n)$, then
\begin{align*}
    \textrm{WF}^{1/2}_\Phi(u) = \kappa^\flat_\varphi(\textrm{WF}^{1/2}(v)),
\end{align*}
where $\kappa^\flat_\varphi$ is as in (\ref{def_kappa_flat_intro}). Given a non-zero metaplectic Fourier integral operator $G$ acting on $H^{-\infty}_\Phi(\C^n)$ whose underlying complex linear canonical transformation $\kappa: \C^{2n} \rightarrow \C^{2n}$ is positive relative to $\Lambda_\Phi$ and such that $\Lambda_\Phi \cap \kappa(\Lambda_\Phi)$ is invariant under $\kappa$, one has the following general relationship between $\textrm{WF}^{1/2}_{\Phi}(Gu)$ and $\textrm{WF}^{1/2}_\Phi(u)$:
\begin{align} \label{wavefront relation in intro}
    \textrm{WF}^{1/2}_\Phi(Gu) = \kappa^\flat(\textrm{WF}^{1/2}_\Phi(u)) \cap \pi_1(\Lambda_\Phi \cap \kappa(\Lambda_\Phi))
\end{align}
where
\begin{align*}
    \kappa^\flat = \pi_1 \circ \kappa \circ (\pi_1|_{\Lambda_\Phi})^{-1}.
\end{align*}
The proof of (\ref{wavefront relation in intro}), which uses the result from \cite{ComplexFIOs} that every metaplectic Fourier integral operator possesses a unique Bergman form, is given in Section 5 below. Using (\ref{wavefront relation in intro}) with $G = e^{-t\tilde{q}^w(z,D_z)}$, $\kappa = \tilde{\kappa}_t$, and the general geometric relationship
\begin{align} \label{geometric relationship for singular space in intro}
\Lambda_\Phi \cap \tilde{\kappa}_t(\Lambda_\Phi) = \kappa_\varphi(S), \ \ \ t>0,
\end{align}
which we establish in Section 6, will complete the proof of Theorem \ref{main theorem of the paper}.

The plan for this paper is as follows. In Section 2, we review background material on metaplectic FBI transforms and exponentially weighted spaces of entire functions. In Section 3, we recall the definition of the $1/2$-Gelfand-Shilov wavefront set and state some basic properties. In Section 4, we discuss metaplectic Fourier integral operators acting on exponentially weighted spaces of entire functions and recall some results from \cite{ComplexFIOs} concerning Bergman representations of such operators. In Section 5, we prove some general results concerning how metaplectic Fourier integral operators in the complex domain move and/or regularize global analytic singularities. Finally, in Section 6, we establish the Bergman representation (\ref{Bergman_form_in_intro}) for the evolution semigroup $e^{-tq^w(x,D)}$ on the FBI transform side and finish the proof of Theorem \ref{main theorem of the paper}.

\vspace{5mm}

{\bf{Acknowledgements}}. Part of this material is based upon work supported by the National Science Foundation under Grant No. DMS-1440140 while the author was in residence at the Mathematics Sciences Research Institute in Berkeley, CA, during the months of November and December of 2019. The author would also like to express his sincerest gratitude to M. Hitrik for offering numerous helpful suggestions and giving feedback on early drafts of this manuscript.

\vspace{5mm}

{\bf{Notation}}
\begin{itemize}
\item \emph{Constants}. $C$ and $c$ stand for positive constants; $C$ may increase from line to line if necessary, while $c$ may decrease from line to line if necessary.
\item \emph{Dot Product}. If $z = (z_1, \ldots, z_n) \in \C^n$ and $w = (w_1, \ldots, w_n) \in \C^n$,
\begin{align*}
    z \cdot w = z_1w_1 + \cdots + z_n \cdot w_n.
\end{align*}
\item $\langle z \rangle = (1+\abs{z}^2)^{1/2}$ denotes the Japanese bracket of $z \in \C^n$.
\item $\pi_1: \C^{2n} \rightarrow \C^n$ is the projection onto the first factor $(z,\zeta) \mapsto z$.
\item $J \in M_{2n \times 2n}(\R)$ is the standard symplectic matrix
\begin{align*}
    J = \begin{pmatrix} 0 & I \\ -I & 0 \end{pmatrix}.
\end{align*}
\item {\emph{Conic Subsets}}. If $X$ is a real vector space with norm $\norm{\cdot}$ and $V, W$ are conic subsets of $X \backslash \{0\}$, we use the notation
\begin{align*}
    V \subset \subset W
\end{align*}
to mean that $V \cap \set{x \in X}{\norm{x}=1}$ is compactly contained in $W \cap \set{x \in X}{\norm{x}=1}$.

\item \emph{Radicals}. If $P$ is a non-negative quadratic form on $\R^N$, then the radical of $P$, denoted $\textrm{Rad}(P)$, is the zero set of $P$, 
\begin{align*}
	\textrm{Rad}(P) = \set{X \in \R^N}{P(X) = 0}.
\end{align*}

\item \emph{Derivatives.} On $\C^n$ with coordinates $z = (z_1, \ldots, z_n)$,
\begin{align*}
    \p_{z_j} = \frac{1}{2} \left(\frac{\p}{\p \textrm{Re} \ z_j} - i \frac{\p}{\p \textrm{Im} \ z_j} \right), \ \ \ j = 1, \ldots, n
\end{align*}
and
\begin{align*}
    \p_{\overline{z}_j} = \frac{1}{2} \left(\frac{\p}{\p \textrm{Re} \ z_j} + i \frac{\p}{\p \textrm{Im} \ z_j} \right), \ \ \ j=1,\ldots, n.
\end{align*}

For a sufficiently differentiable function $f: \C^n \rightarrow \C$,
\begin{align*}
    \p_zf = \begin{pmatrix} \frac{\p f}{\p z_1} \\ \vdots \\ \frac{\p f}{\p z_n} \end{pmatrix}, \ \ \  \p_{\overline{z}} f = \begin{pmatrix} \frac{\p f}{\p \overline{z}_1} \\ \vdots \\ \frac{\p f}{\p \overline{z}_n} \end{pmatrix}
\end{align*}
and
\begin{align*}
    \p^2_{zz}f = \left(\frac{\p^2 f}{\p z_j \p z_k} \right)_{1 \le j, k \le n}, \ \ \ \p^2_{z\overline{z}}f = \left(\frac{\p^2 f}{\p z_j \p \overline{z}_k} \right)_{1 \le j, k \le n}, \\
    \p^2_{\overline{z}z}f = \left(\frac{\p^2 f}{\p \overline{z}_j \p z_k} \right)_{1 \le j, k \le n,}, \ \ \ \p^2_{\overline{z}\overline{z}}f = \left(\frac{\p^2 f}{\p \overline{z}_j \p \overline{z}_k} \right)_{1 \le j, k \le n}.
\end{align*}

\item \emph{Hamiltonian Dynamics in the Complex Domain}. Write $\C^{2n} = \C^n_z \times \C^n_\zeta$ and let $\sigma = d\zeta \wedge dz$ be the standard complex symplectic form on $\C^{2n}$. Given $f \in \textrm{Hol}(\C^{2n})$, the complex Hamilton vector field of $f$ is denoted
\begin{align*}
    H_f = \p_\zeta f \cdot \p_z - \p_z f \cdot \p_\zeta,
\end{align*}
This is a complex vector field on $\C^{2n}$ of type $(1,0)$, which we identify with $(\p_\zeta f, -\p_z f) \in \C^{2n}$. We note that $H_f$ is the unique complex vector field on $\C^{2n}$ of type $(1,0)$ such that $\sigma(t, H_f) = df(t)$ for all complex vector fields $t$ on $\C^{2n}$ of type $(1,0)$. In this paper we shall only ever be interested in the case where $f$ is a holomorphic quadratic form. In this case the Hamilton flow of $f$ is denoted by $\exp{(tH_f)}$, $t \in \R$, and
defined as follows: for $(z_0, \zeta_0) \in \C^{2n}$ and $t \in \R$, 
\begin{align*}
    (z(t),\zeta(t)) = \exp{(tH_f)}(z_0, \zeta_0)
\end{align*}
if and only if
\begin{align*}
\begin{cases}
z'(t) = \p_\zeta f(z(t), \zeta(t)) \\
\zeta'(t) = -\p_z f(z(t), \zeta(t)), \\
z(0) = z_0, \ \zeta(0) = \zeta_0.
\end{cases}
\end{align*}

\item \emph{Schwartz Functions and Tempered Distributions}. $\mathcal{S}(\R^n)$ denotes class of Schwartz functions on $\R^n$ equipped with its usual locally convex Frech\'{e}t topology; $\mathcal{S}'(\R^n)$ is the space of tempered distributions on $\R^n$ endowed with the weak* topology. The distributional pairing of $u \in \mathcal{S}'(\R^n)$ and $f \in \mathcal{S}(\R^n)$ is denoted $\langle u, f \rangle$.

\end{itemize}


\section{Review of FBI Tools and Exponentially Weighted Spaces of Entire Functions}

In this section we review basic definitions and facts concerning metaplectic Fourier-Bros-Iagolnitzer (FBI) transforms on $\R^n$. Standard references for this material include \cite{Minicourse}, Chapter 13 of \cite{SemiclassicalAnalysis}, and \cite{GlobalILagrangians}.

Let $\varphi = \varphi(z,y)$ be a holomorphic quadratic form on $\C^{2n} = \C^n_z \times \C^n_y$. Write
\begin{align*}
    \varphi(z,y) = \frac{1}{2} A z \cdot z + B z \cdot y + \frac{1}{2} D y \cdot y, \ \ \ (z,y) \in \C^n \times \R^n,
\end{align*}
where $A, B, D \in M_{n \times n}(\C)$ with $A = A^T$ and $D = D^T$. If $\det{B} \neq 0$ and $\textrm{Im} \ D$ is positive-definite, we say that $\varphi$ is an \emph{FBI phase function} or \emph{metaplectic FBI phase function}. If $\varphi$ is an FBI phase function, the FBI transform associated to $\varphi$ is the linear transformation $\mathcal{T}_\varphi: \mathcal{S}'(\R^n) \rightarrow \textrm{Hol}(\C^n)$
defined by
\begin{align} \label{definition of FBI transform}
    \mathcal{T}_\varphi u(z) = c_\varphi \int_{\R^n} e^{i \varphi(z,y)} u(y) \, dy, \ \ \ u \in \mathcal{S}'(\R^n),
\end{align}
where
\begin{align} \label{choice of C to make FBI transform unitary}
    c_\varphi = 2^{-n/2} \pi^{-3n/4} (\det{\textrm{Im} \ D})^{-1/4} \abs{\det{B}},
\end{align}
and the integral (\ref{definition of FBI transform}) is interpreted in the sense of distributions.

Let $\varphi$ be an FBI phase function. To describe the range of the FBI transform $\mathcal{T}_\varphi$, we introduce the real-valued quadratic form
\begin{align} \label{quadratic form associated to an FBI phase}
    \Phi(z) = \sup_{y \in \R^n} (-\textrm{Im} \ \varphi(z,y)), \ \ \ z \in \C^n.
\end{align}
Since $\textrm{Im} \ D$ is positive-definite, this supremum is really a maximum and we may write
\begin{align*}
    \Phi(z) = -\textrm{Im} \ \varphi(z,y(z))
\end{align*}
where $y(z) \in \R^n$ is an $\R$-linear function of $z \in \C^n$. Because $\Phi$ is equal to the maximum of the family of pluriharmonic functions
\begin{align*}
    \C^n \ni z \mapsto -\textrm{Im} \ \varphi(z,y) \in \R, \ \ \ y \in \R^n,
\end{align*}
the form $\Phi$ is itself plurisubharmonic. In fact, $\Phi$ is strictly plurisubharmonic, i.e. the Levi matrix $\p^2_{\overline{z} z} \Phi$ is Hermitian positive-definite. We refer to Proposition 1.3.2 of \cite{Minicourse} for a proof. In the sequel, if $\varphi$ is an FBI phase function, then we shall we refer to $\Phi$ given by (\ref{quadratic form associated to an FBI phase}) as the strictly plurisubharmonic weight associated to $\varphi$.

For $s \in \R$, let
\begin{align}
	L^2_{\Phi, s}(\C^n) = L^2(\C^n, \langle z \rangle^{2s} e^{-2\Phi(z)} \, L(dz)),
\end{align}
equipped with the natural inner product,
\begin{align}
	(u_1, u_2)_s = \int_{\C^n} u_1(z) \overline{u_2(z)} \langle z \rangle^{2s} e^{-2 \Phi(z)} \, L(dz), \ \ u_1, u_2 \in L^2_{\Phi, s}(\C^n),
\end{align}
and associated norm
\begin{align} \label{s_norm}
	\norm{u}^2_s = \int_{\C^n} \abs{u(z)}^2 \langle z \rangle^{2s} e^{-2\Phi(z)} \, L(dz), \ \ u \in L^2_{\Phi, s}(\C^n).
\end{align}
For $s \in \R$, let
\begin{align}
	H^s_\Phi(\C^n) = L^2_{\Phi, s}(\C^n) \cap \textrm{Hol}(\C^n)
\end{align}
be the closed linear subspace of entire functions in $L^2_{\Phi,s}(\C^n)$. By convention, when $s = 0$, we write $H_\Phi(\C^n)$ in place of $H^0_\Phi(\C^n)$, and we write $(\cdot, \cdot)$ and $\norm{\cdot}$ in place of $(\cdot, \cdot)_0$ and $\norm{\cdot}_0$ respectively. In the literature, the space $H_{\Phi}(\C^n)$ is known as the \emph{Bargmann} or \emph{Bargmann-Fock} space of entire functions on $\C^n$. If $s_1 \le s_2$, then $H^{s_2}_\Phi(\C^n) \subset H^{s_1}_\Phi(\C^n)$, and the natural inclusion map $H^{s_2}_\Phi(\C^n) \hookrightarrow H^{s_1}_\Phi(\C^n)$ is bounded. Thus, if we let
\begin{align} \label{extended_Bargmann_space}
	H^{-\infty}_\Phi(\C^n) = \bigcup_{s \in \R} H^s_\Phi(\C^n),
\end{align}
we obtain an inductive system of Hilbert spaces $(H^{-\infty}_\Phi(\C^n), \{H^s_\Phi(\C^n)\}_{s \in \R})$ (see \cite{Conway} Chapter IV Section 5). We equip $H^{-\infty}_\Phi(\C^n)$ with the corresponding inductive limit topology (see \cite{Conway} Chapter IV Proposition 5.3 and Definition 5.4). We refer to $H^{-\infty}_\Phi(\C^n)$ as the \emph{extended Bargmann space}. We also introduce the space
\begin{align}
	H^\infty_\Phi(\C^n) = \bigcap_{s \in \R} H^s_\Phi(\C^n),
\end{align}
equipped with the Frech\'{e}t space topology induced by the family of norms $\{\norm{\cdot}_s\}_{s \in \R}$. For every $s \in \R$, we have continuous inclusions
\begin{align}
	H^\infty_\Phi(\C^n) \hookrightarrow H^s_\Phi(\C^n) \hookrightarrow H^{-\infty}_\Phi(\C^n).
\end{align}
Thanks to the mean-value property of holomorphic functions, we have the following lemma characterizing functions in $H^s_\Phi(\C^n)$ as holomorphic functions on $\C^n$ obeying suitable weighted $L^\infty$-estimates.

\begin{lemma} \label{growth_of_entire_functions}
	For any $s \in \R$ and $\epsilon>0$, there is $C>0$ such that
	\begin{align}
		\norm{u}_{L^2_{\Phi,s}(\C^n)} \le C \norm{\langle z \rangle^{s+n+\epsilon} u(z) e^{-\Phi(z)}}_{L^\infty(\C^n)}
	\end{align}
	for all measurable $u$ on $\C^n$. For any $N \in \R$, there is $C>0$ such that
	\begin{align} \label{second_bound_in_lemma}
		\norm{\langle z \rangle^{n+N} u(z) e^{-\Phi(z)}}_{L^\infty(\C^n)} \le C \norm{u}_{L^2_{\Phi, N+2n} (\C^n)}
	\end{align}
	for all $u \in \textrm{Hol}(\C^n)$. Consequently, for $u \in \textrm{Hol}(\C^n)$,
	\begin{align}
	u \in H^\infty_\Phi(\C^n) \iff \forall N\in \R \ \exists C>0 \ \forall z \in \C^n: \abs{u(z)} \le C \langle z \rangle^N e^{\Phi(z)}
	\end{align}
	and
	\begin{align}
		u \in H^{-\infty}_\Phi(\C^n) \iff \exists N\in \R \ \exists C>0 \ \forall z \in \C^n: \abs{u(z)} \le C \langle z \rangle^N e^{\Phi(z)}.
	\end{align}
\end{lemma}
\begin{proof}
	Let $s \in \R$ and $\epsilon>0$ be arbitrary. For any $N \in \R$ and measurable $u$ on $\C^n$,
	\begin{align}
		\norm{u}_{L^2_{\Phi,s}(\C^n)} \le \left(\int_{\C^n} \langle z \rangle^{2(s-N)} \, L(dz) \right)^{1/2} \norm{\langle z \rangle^N u(z) e^{-\Phi(z)}}_{L^\infty(\C^n)}.
	\end{align}
	Choosing $N = s + n + \epsilon$ gives
	\begin{align}
		\norm{u}_{L^2_{\Phi,s}(\C^n)} \le C \norm{\langle z \rangle^{s+n+\epsilon} u(z) e^{-\Phi(z)}}_{L^\infty(\C^n)},
	\end{align}
	where the constant $C>0$ depends only on $s$, $n$, and $\epsilon$. To prove the second claim, let $N \in \R$ and $u \in \textrm{Hol}(\C^n)$ be arbitrary. By the mean-value theorem,
	\begin{align}
		u(z) = \frac{1}{c_{2n}} \langle z \rangle^{2n} \int_{\abs{z-w} \le \langle z \rangle^{-1}} u(w) \, L(dw), \ \ z \in \C^n,
	\end{align}
	where $c_{2n}$ is the volume of the unit ball in $\C^n \cong \R^{2n}$. Observe that there is $C>0$ such that
	\begin{align}
		\abs{\Phi(w)-\Phi(z)} \le C
	\end{align}
	whenever $z, w \in \C^n$ are such that $\abs{z-w} \le \langle z \rangle^{-1}$. Thus 
	\begin{align}
		\begin{split}
		\langle z \rangle^N \abs{u(z)} e^{-\Phi(z)} &\le C \int_{\abs{z-w} \le \langle z \rangle^{-1}} \langle z \rangle^{N+2n} \abs{u(w)} e^{-\Phi(w)} \, L(dw) \\
		&\le C \int_{\abs{z-w} \le \langle z \rangle^{-1}} \langle z -w \rangle^{\abs{N+2n}} \langle w \rangle^{N+2n} \abs{u(w)} e^{-\Phi(w)} \, L(dw) \\
		&\le C \int_{\abs{z-w} \le \langle z \rangle^{-1}} \langle w \rangle^{N+2n} \abs{u(w)} e^{-\Phi(w)} \, L(dw) \\
		&\le C \langle z \rangle^{-n} \norm{u}_{H^{N+2n}_\Phi}(\C^n)
		\end{split}
	\end{align}
	for all $z \in \C^n$, where the constant $C>0$ does not depend on $u$. The bound (\ref{second_bound_in_lemma}) follows.
\end{proof}

Next, we state a well-known proposition characterizing the range of a metaplectic FBI transform $\mathcal{T}_\varphi$ in terms of exponentially weighted spaces of entire functions on $\C^n$.

\begin{proposition} \label{image_of_FBI_transform}
	Let $\varphi$ be an FBI phase function with associated FBI transform $\mathcal{T}_\varphi$ and strictly plurisubharmonic weight $\Phi$. Then
	\begin{align}
		\mathcal{T}_\varphi: L^2(\R^n) \rightarrow H_\Phi(\C^n)
	\end{align}
	is a unitary transformation. Furthermore, $\mathcal{T}_\varphi$ is bijective
	\begin{align}
		\mathcal{S}(\R^n) \rightarrow H^\infty_\Phi(\C^n) \ \ \textrm{and} \ \ \mathcal{S}'(\R^n) \rightarrow H^{-\infty}_\Phi(\C^n).
	\end{align}
\end{proposition}
\begin{proof}
	There are many available proofs of the unitarity of $\mathcal{T}_\varphi: L^2(\R^n) \rightarrow H_\Phi(\C^n)$. The reader may consult, for instance, Theorem 13.7 of \cite{SemiclassicalAnalysis}, Theorem 1.3.3 of \cite{Minicourse}, or Proposition 6.1 of \cite{HyperbolicOperators}. The bijectivity of $\mathcal{T}_\varphi: \mathcal{S}(\R^n) \rightarrow H^\infty_\Phi(\C^n)$ and $\mathcal{T}_\varphi: \mathcal{S}'(\R^n) \rightarrow H^{-\infty}_\Phi(\C^n)$ is also well-known. It follows immediately from, for instance, Lemma \ref{growth_of_entire_functions} and Proposition 6.1 of \cite{HyperbolicOperators}. See also Section 12.2 of \cite{Lectures_on_Resonances}.
\end{proof}

We next discuss some functional analytic aspects of the spaces $H^s_\Phi(\C^n)$, $s\in \R$. Let $\Phi$ be a strictly plurisubharmonic quadratic form on $\C^n$. Let $\Psi(\cdot, \cdot)$ be the polarization of $\Phi$, i.e. the unique holomorphic quadratic form on $\C^{2n}$ such that $\Psi(z,\overline{z}) = \Phi(z)$ for all $z \in \C^n$. We recall (see \cite{Minicourse} or Chapter 13 of \cite{SemiclassicalAnalysis}) that the \emph{Bergman projection} associated to $\Phi$ is the orthogonal projector $\Pi_\Phi: L^2_{\Phi, 0}(\C^n) \rightarrow H_\Phi(\C^n)$. It is given explicitly by
\begin{align} \label{definition_Bergman_projector}
	\Pi_\Phi u(z) = C_\Phi \int_{\C^n} e^{2 \Psi(z, \overline{w})} u(w) e^{-2 \Phi(w)} \, L(dw), \ \ u \in H_\Phi(\C^n),
\end{align}
where
\begin{align} \label{Bergman_constant}
	C_\Phi = \left(\frac{2}{\pi} \right)^n \det{\p^2_{z \overline{z}} \Phi}.
\end{align}
We recall that $\Psi(\cdot, \cdot)$ satisfies the `fundamental estimate'
\begin{align} \label{fundamental_estimate}
	2 \textrm{Re} \ \Psi(z,\overline{w}) - \Phi(z) - \Phi(w) \asymp -\abs{z-w}^2, \ \ z ,w \in \C^n.
\end{align}
 For a proof, see page 492 of \cite{Minicourse} or the proof of Theorem 13.6 in \cite{SemiclassicalAnalysis}. From (\ref{fundamental_estimate}) and Schur's lemma, it follows that the operator (\ref{definition_Bergman_projector}) is bounded $L^2_{\Phi, s}(\C^n) \rightarrow L^2_{\Phi, s}(\C^n)$ for every $s \in \R$. In fact, $\Pi_\Phi$ defined by (\ref{definition_Bergman_projector}) coincides with the orthogonal projection $L^2_{\Phi, s}(\C^n) \rightarrow H^s_\Phi(\C^n)$ (see \cite{Lectures_on_Resonances} Section 12.2), and we have
\begin{align} \label{Bergman_projection_is_reproducing}
	\Pi_\Phi u = u, \ \ \forall u \in H^{-\infty}_\Phi(\C^n).
\end{align}
Using (\ref{Bergman_projection_is_reproducing}), one can prove

\begin{proposition}[\cite{Lectures_on_Resonances} Section 12.2] \label{density_proposition} For any strictly plurisubharmonic quadratic form $\Phi$ on $\C^n$, the space $H^{\infty}_\Phi(\C^n)$ is dense in $H^s_\Phi(\C^n)$ for every $s \in \R$. Consequently, $H^\infty_\Phi(\C^n)$ is dense in $H^{-\infty}_\Phi(\C^n)$.
\begin{proof}
Let $\Phi$ be a strictly plurisubharmonic quadratic form on $\C^n$. Let $s \in \R$ and $u \in H^s_\Phi(\C^n)$ be arbitrary. Let $\chi \in C^\infty_0(\C^n)$ be such that $0 \le \chi \le 1$ and $\chi \equiv 1$ in a neighborhood of $0 \in \C^n$. For $\epsilon>0$, set
\begin{align}
	u_\epsilon = \Pi_\Phi(\chi(\epsilon z) u),
\end{align}	
where $\Pi_\Phi$ is the Bergman projector (\ref{definition_Bergman_projector}). Using the fundamental estimate (\ref{fundamental_estimate}), we see that for any $\epsilon>0$ and $N>0$,
\begin{align}
	\begin{split}
	\langle z \rangle^N \abs{u_\epsilon(z)} e^{-\Phi(z)} &\le C \int_{\C^n} e^{-c \abs{z-w}^2} \langle z \rangle^N \abs{\chi(\epsilon w) u(w)} e^{-\Phi(w)} \, L(dw) \\
	&\le C \int_{\C^n} e^{-c\abs{z-w}^2} \langle z - w \rangle^N \langle w \rangle^N \abs{\chi(\epsilon w) u(w)} e^{-\Phi(w)} \, L(dw) \\
	&\le C_{\epsilon, N} \int_{\C^n} e^{-c \abs{z-w}^2} \langle z - w \rangle^{N} \langle w \rangle^s \abs{u(w)} e^{-\Phi(w)} \, L(dw),
	\end{split}
\end{align}
where $C_{\epsilon, N}>0$ depends only on $\epsilon$ and $N$. Applying Schur's lemma, we find that $u_\epsilon \in H^N_\Phi(\C^n)$ for any $\epsilon>0$ and $N>0$. Thus $u_\epsilon \in H^\infty_\Phi(\C^n)$ for all $\epsilon>0$. Now we claim that $u_\epsilon \rightarrow u$ in $H^s_\Phi(\C^n)$ as $\epsilon \rightarrow 0^+$. Indeed, from (\ref{Bergman_projection_is_reproducing}),
\begin{align}
	\begin{split}
	\langle z \rangle^s \abs{u(z)-u_\epsilon(z)} e^{-\Phi(z)} &\le C \int_{\C^n} e^{-c \abs{z-w}^2} \langle z \rangle^s (1-\chi(\epsilon w)) \abs{u(w)} e^{-\Phi(w)} \, L(dw) \\
	&\le C \int_{\C^n} K_s(z,w) \langle w \rangle^s (1-\chi(\epsilon w)) \abs{u(w)} e^{-\Phi(w)} \, L(dw),
	\end{split}
\end{align}
where
\begin{align}
	K_s(z,w) = \langle z - w \rangle^{\abs{s}} e^{-c \abs{z-w}^2}, \ \ z, w \in \C^n.
\end{align}
As
\begin{align}
	\sup_{z \in \C^n} \int_{\C^n} \abs{K_s(z,w)} \, L(dw) < \infty \ \ \textrm{and} \ \ \sup_{w \in \C^n} \int_{\C^n} \abs{K_s(z,w)} \, L(dz) < \infty,
\end{align}
we deduce from Schur's lemma that there is a $C>0$ such that
\begin{align} \label{almost_done_with_lemma}
	\norm{u-u_\epsilon}_{H^s_\Phi(\C^n)} \le C \norm{(1-\chi(\epsilon z)) u(z)}_{L^2_{\Phi, s}(\C^n)}
\end{align}
for all $\epsilon>0$. By dominated convergence, the righthand side of (\ref{almost_done_with_lemma}) converges to $0$ as $\epsilon \rightarrow 0^+$. Therefore $\norm{u-u_\epsilon}_{H^s_\Phi(\C^n)} \rightarrow 0$ as $\epsilon \rightarrow 0^+$.
\end{proof}
\end{proposition}

We conclude this section by establishing a proposition that identifies the dual space of $H^s_{\Phi}(\C^n)$ with $H^{-s}_\Phi(\C^n)$.

\begin{proposition} \label{dual_space_proposition}
	Let $\Phi$ be a strictly plurisubharmonic quadratic form on $\C^n$ and let $s \in \R$. For every $v \in H^{-s}_\Phi(\C^n)$, the functional 
	\begin{align} \label{psi_v}
		\psi_v(u) = \int_{\C^n} u(z) \overline{v(z)} e^{-2\Phi(z)} \, L(dz), \ \ u \in H^s_\Phi(\C^n), 
	\end{align}
	defines an element of $(H^s_\Phi(\C^n))'$ with $\norm{\psi_v} \le \norm{v}_{-s}$. Moreover, the map $v \mapsto \psi_v$ is a bounded antilinear isomorphism $H^{-s}_\Phi(\C^n) \rightarrow (H^s_\Phi(\C^n))'$.
\end{proposition}
\begin{proof}
	Let $s \in \R$ be fixed. For $v \in H^{-s}_\Phi(\C^n)$, let $\psi_v$ be as in (\ref{psi_v}). An application of the Cauchy-Schwarz inequality gives
	\begin{align}
		\abs{\psi_v(u)} \le \norm{v}_{-s} \norm{u}_s
	\end{align}
	for all $u \in H^s_\Phi(\C^n)$, and so it is clear that the map $v \mapsto \psi_v$ is a bounded antilinear mapping $H^{-s}_\Phi(\C^n) \rightarrow (H^s_\Phi(\C^n))'$. To see that the map $v \mapsto \psi_v$ is injective, let $v \in H^{-s}_\Phi(\C^n)$ be such that $\psi_v \equiv 0$ on $H^s_\Phi(\C^n)$. Then, using (\ref{Bergman_projection_is_reproducing}) and the identity
	\begin{align} \label{Bergman_phase_symmetry}
		\Psi(w, \overline{z}) = \overline{\Psi(z,\overline{w})}, \ \ w, z \in \C^n,
	\end{align}
	(see Lemma 13.1 in \cite{SemiclassicalAnalysis}), we get, for any $u \in L^2_{\Phi, s}(\C^n)$,
	\begin{align}
		\begin{split}
		&\int_{\C^n} u(z) \overline{v(z)} e^{-2\Phi(z)} \, L(dz) = \int_{\C^n} u(z) \overline{\Pi_\Phi v(z)} e^{-2\Phi(z)} \, L(dz) \\
		&= \int_{\C^n} \Pi_\Phi u(z) \overline{v(z)} e^{-2\Phi(z)} \, L(dz) = 0.
		\end{split}
	\end{align}
	In particular,
	\begin{align}
		\int_{\C^n} u(z) \overline{v(z)} e^{-2\Phi(z)} \, L(dz) = 0
	\end{align}
	for all $u \in C^\infty_0(\C^n)$, and we deduce $v = 0$. To prove surjectivity, let $\psi \in (H^s_\Phi(\C^n))'$ be arbitrary. By the Riesz representation theorem, there exists a unique $v_1 \in H^s_{\Phi}(\C^n)$ such that
	\begin{align}
		\psi(u) = \int_{\C^n} u(z) \overline{v_1(z)} \langle z \rangle^{2s} e^{-2\Phi(z)} \, L(dz), \ \ u \in H^s_\Phi(\C^n).
	\end{align}
	Let
	\begin{align}
		v = \Pi_\Phi(v_1(z) \langle z \rangle^{2s}).
	\end{align}
	Observing that
	\begin{align}
		v_1(z) \langle z \rangle^{2s} \in L^2_{\Phi, -s}(\C^n),
	\end{align}
	we see that $v \in H^{-s}_\Phi(\C^n)$ and hence
	\begin{align}
		\psi = \psi_v.
	\end{align}
	By the closed graph theorem, the map $v \mapsto \psi_v$ is a bounded antilinear isomorphism $H^{-s}_\Phi(\C^n) \rightarrow (H^s_\Phi(\C^n))'$.
\end{proof}


\section{The $1/2$-Gelfand-Shilov Wavefront Set}

In this section we recall the basic definition and properties of the $1/2$-Gelfand-Shilov wavefront set. For a full discussion, see \cite{GelfandShilov} or \cite{HyperbolicOperators}.

Let $\C^{2n} = \C^n_z \times \C^n_\zeta$ and let $\sigma = d\zeta \wedge dz$ be the standard complex symplectic form on $\C^{2n}$. When equipped with the form $\sigma$, the space $\C^{2n}$ becomes a complex symplectic vector space. A $\C$-linear map $\kappa: \C^{2n} \rightarrow \C^{2n}$ is said to be a \emph{complex linear symplectomorphism} or \emph{complex linear canonical transformation} if $\kappa^* \sigma = \sigma$. Note that if $\kappa: \C^{2n} \rightarrow \C^{2n}$ is a complex linear canonical transformation, then $\det{\kappa}=1$, and hence $\kappa$ is automatically bijective.

Associated to the complex symplectic form $\sigma$ are the real $2$-forms $\textrm{Re} \ \sigma$ and $\textrm{Im} \ \sigma$. Suppose $\Sigma$ is a real linear subspace of $\C^{2n}$. We say $\Sigma$ is \emph{$R$-symplectic} if the restriction of $\textrm{Re} \ \sigma$ to $\Sigma$ is non-degenerate, and we say that $\Sigma$ is \emph{$I$-Lagrangian} if $\Sigma$ is a Lagrangian subspace of $\C^{2n}$ with respect to $\textrm{Im} \ \sigma$, i.e. $\textrm{dim}_\R{\Sigma} = 2n$ and $\textrm{Im} \ \sigma|_{\Sigma} = 0$. We say that $\Sigma$ is \emph{totally real} if $\Sigma \cap i \Sigma = \{0\}$, and if, in addition $\textrm{dim}_{\R} \Sigma = 2n$, we say $\Sigma$ is \emph{maximally totally real}. Any real linear subspace of $\C^{2n}$ that is $I$-Lagrangian and $R$-symplectic is automatically maximally totally real. For more background on complex symplectic linear algebra, see \cite{Minicourse}.

Let $\varphi$ be an FBI phase function. By Proposition 1.3.2 of \cite{Minicourse}, $\varphi$ generates a complex linear canonical transformation $\kappa_\varphi: \C^{2n} \rightarrow \C^{2n}$ given implicitly by
\begin{align} \label{canonical transformation associated to FBI phase}
    \C^{2n} \ni (y, - \p_y\varphi(z,y)) \mapsto (z, \p_z \varphi(z,y)) \in \C^{2n}, \ \ (z,y) \in \C^{2n}.
\end{align}
It can be shown that the $\kappa_\varphi$ maps $\R^{2n}$ bijectively onto the space
\begin{align} \label{IR subspace associated to Phi}
    \Lambda_\Phi = \set{\left(z, \frac{2}{i} \p_z \Phi(z) \right) \in \C^{2n}}{z \in \C^n},
\end{align}
where $\Phi$ is the strictly plurisubharmonic weight associated to $\varphi$. See \cite{SemiclassicalAnalysis} Theorem 13.5 for a proof. Since $\kappa_\varphi$ is a complex linear canonical transformation, the space $\Lambda_\Phi$ is $I$-Lagrangian and $R$-symplectic. Hence $\Lambda_\Phi$ is maximally totally real.

Suppose $\Phi$ is a strictly plurisubharmonic quadratic form on $\C^n$. Let $\Lambda_\Phi$ be as in (\ref{IR subspace associated to Phi}) and let $\textrm{pr}_\Phi = \pi_1|_{\Lambda_{\Phi}}$. Since $\Lambda_\Phi$ is the graph of the $\R$-linear map $\frac{2}{i} \p_z \Phi: \C^{n} \rightarrow \C^n$, the projection $\textrm{pr}_{\Phi}$ is an $\R$-linear isomorphism $\Lambda_\Phi \rightarrow \C^n$. If $\Phi$ is the strictly plurisubharmonic weight associated to an FBI phase function $\varphi$ and $\kappa_\varphi$ is the complex canonical transformation (\ref{canonical transformation associated to FBI phase}) generated by $\varphi$, then the composition
\begin{align} \label{first definition of kappa flat}
    \kappa^\flat_\varphi = \textrm{pr}_\Phi \circ \kappa_\varphi
\end{align}
is an $\R$-linear isomorphism $\mathbb{R}^{2n} \rightarrow \C^n$.

Let $X$ be a real vector space. A subset $V$ of $X \backslash \{0\}$ is said to be \emph{conic} if $tx \in V$ whenever $x \in V$ and $t > 0$. If $X$ and $Y$ are real vector spaces, $T: X \rightarrow Y$ is a linear map, and $V$ is a conic subset of $X \backslash \{0\}$, then $T(V)$ is a conic subset of $Y$. In particular, if $V$ is a conic subset of $\R^{2n} \backslash \{0\}$, $\varphi$ is an FBI phase function, and $\kappa^\flat_\varphi$ is as in (\ref{first definition of kappa flat}), then $\kappa^\flat_\varphi(V)$ is a conic subset of $\C^n \backslash \{0\}$.

\begin{definition}[\cite{GelfandShilov} Definition 3.1, \cite{HyperbolicOperators} Definition 6.6] Let $u \in \mathcal{S}'(\R^n)$. The \emph{1/2-Gelfand-Shilov wavefront} of $u$, denoted $\textrm{WF}^{1/2}(u)$, is the complement in $\R^{2n} \backslash \{(0,0)\}$ of the set of points $(x_0, \xi_0)$ for which there exists an FBI phase function $\varphi$ and constants $C, c > 0$ such that
\begin{align} \label{complement of analytic Gaussian decay estimate}
    \abs{\mathcal{T}_\varphi u(z)} \le C e^{\Phi(z)-c \abs{z}^2}
\end{align}
for all $z$ within some open conic neighborhood $V$ of $\kappa^\flat_\varphi(x_0, \xi_0)$ in $\C^n \backslash \{0\}$. Here $\Phi$ is the strictly plurisubharmonic weight associated to $\varphi$ and $\kappa^\flat_\varphi$ is as in (\ref{first definition of kappa flat}).
\end{definition}

As shown in  \cite{HyperbolicOperators} Proposition 6.4, one may use any FBI phase to determine if a point $(x_0, \xi_0) \in \R^{2n} \backslash \{(0,0)\}$ lies in $\textrm{WF}^{1/2}(u)$. Indeed, suppose $u \in \mathcal{S}'(\R^n)$, $(x_0, \xi_0) \in \R^{2n} \backslash \{(0,0)\}$, and that there is an FBI phase function $\varphi$, constants $C, c>0$, and an open conic neighborhood $V$ of $\kappa^\flat_\varphi(x_0, \xi_0)$ in $\mathbb{C}^n \backslash \{0\}$ such that the estimate (\ref{complement of analytic Gaussian decay estimate}) holds in $V$. If $\varphi_1$ is another FBI phase with associated strictly plurisubharmonic weight $\Phi_1$ and associated canonical transformation $\kappa_{\varphi_1}$, then there exists an open conic neighborhood $V_1$ of $\kappa^\flat_{\varphi_1}(x_0, \xi_0)$ in $\C^{n} \backslash \{0\}$ and constants $C_1, c_1>0$ such that
\begin{align*}
    \abs{\mathcal{T}_{\varphi_1} u(z)} \le C_1 e^{\Phi_1(z)-c \abs{z}^2}, \ \ \ z \in V_1.
\end{align*}
Thus we may reformulate the definition of $\textrm{WF}^{1/2}(u)$ as follows: if $u \in \mathcal{S}'(\R^n)$ and $(x_0, \xi_0) \in \R^{2n} \backslash \{(0,0)\}$, then $(x_0, \xi_0) \in \textrm{WF}^{1/2}(u)$ if and only if there is an FBI phase function $\varphi$ such that for every open conic neighborhood $V$ of $\kappa^\flat_\varphi(x_0, \xi_0)$ and every choice of constants $C, c >0$ the estimate (\ref{complement of analytic Gaussian decay estimate}) fails to hold for every $z \in V$. In particular, $(x_0, \xi_0) \in \textrm{WF}^{1/2}(u)$ if there is an FBI phase $\varphi$ with associated weight $\Phi$ and canonical transformation $\kappa_\varphi$ such that
\begin{align*}
    \mathcal{T}_\varphi u(\lambda \kappa^\flat_\varphi(x_0,\xi_0)) e^{-\Phi(\lambda \kappa_\varphi(x_0,\xi_0))} \neq o(1) \ \textrm{as} \ \lambda \rightarrow \infty.
\end{align*}

For notational purposes, it is convenient to introduce the notion of the $1/2$-Gelfand-Shilov wavefront set of elements of $H^{-\infty}_\Phi(\C^n)$ relative to the plurisubharmonic weight $\Phi$. This definition appears to be original, but we feel its addition will help to make some of our results more transparent.

\begin{definition} \label{relative analytic wavefront set}
Let $\Phi$ be a strictly plurisubharmonic quadratic form on
$\C^n$. For $u \in H^{-\infty}_\Phi(\C^n)$, we define \emph{$1/2$-Gelfand-Shilov wavefront set of $u$ relative to $\Phi$}, denoted $\textrm{WF}^{1/2}_\Phi(u)$, as the complement in $\C^{n}\backslash \{0\}$ of all $z \in \C^n \backslash \{0\}$ for which there exists an open conic neighborhood $V$ of $z$ in $\C^n \backslash \{0\}$ and constants $C, c>0$ such that
\begin{align*}
    \abs{u(z)} \le C e^{\Phi(z)-c \abs{z}^2}, \ \ \ z \in V.
\end{align*}
\end{definition}

If $u \in H^{-\infty}_\Phi(\C^n)$, then $\textrm{WF}^{1/2}_\Phi(u)$ is a closed conic subset of $\C^n \backslash \{0\}$. Furthermore, using Definition \ref{relative analytic wavefront set}, we may restate the definition of the $1/2$-Gelfand-Shilov wavefront set of tempered distributions as follows: if $u \in \mathcal{S}'(\R^n)$, then a point $(x_0, \xi_0) \in \R^{2n} \backslash \{(0,0)\}$ belongs to $\textrm{WF}^{1/2}(u)$ if and only if there is an FBI phase $\varphi$ with associated weight $\Phi$ and canonical transformation $\kappa_\varphi$ such that $\kappa^\flat_\varphi(x_0, \xi_0) \in \textrm{WF}^{1/2}_\Phi(\mathcal{T}_\varphi u)$. If $u \in \mathcal{S}'(\R^n)$, $(x_0, \xi_0) \in \R^{2n}$, and $\varphi$ is an FBI phase such that $\kappa^\flat_\varphi(x_0,\xi_0) \in \textrm{WF}^{1/2}_\Phi(\mathcal{T}_\varphi u)$, then the same is true with $\varphi$ replaced by any other FBI phase $\varphi_1$ and with $\Phi$ and $\kappa_\varphi$ replaced by the strictly plurisubharmonic weight and complex canonical transformation associated to $\varphi_1$ respectively.

\begin{proposition}[\cite{HyperbolicOperators} \label{empty analytic wavefront set} Proposition 6.9]
Suppose $u \in \mathcal{S}'(\R^n)$. Then $\textrm{WF}^{1/2}(u) = \emptyset$ if and only if $u$ extends to a holomorphic function $U(z)$ on $\C^n$ such that
\begin{align} \label{Gelfand-Shilov estimate}
    \abs{U(z)} \le C e^{C \abs{y}^2-c\abs{x}^2}
\end{align}
for all $z = x + i y \in \C^n$.
\end{proposition}

A function $u \in C^\infty(\R^n)$ that extends to an entire analytic function on $\C^n$ and whose extension satisfies the estimate (\ref{Gelfand-Shilov estimate}) is known as a \emph{Gelfand-Shilov test function}. For more information on the Gelfand-Shilov space of test functions and its topological dual, we refer the reader to \cite{GelfandShilov}. 

Proposition \ref{empty analytic wavefront set} is significant because it shows that the $1/2$-Gelfand-Shilov wavefront of a distribution $u$ captures our intuition for what the set of global analytic singularities of a tempered distribution $u$ should be. Namely, for $\textrm{WF}^{1/2}(u)$ to be empty, $u$ must not only be real analytic, but also satisfy a Gaussian type decay estimate as $\abs{x} \rightarrow \infty$.

An elementary compactness argument also allows us to deduce a necessary and sufficient condition so that the relative $1/2$-Gelfand-Shilov wavefront set of $u \in H^{-\infty}_\Phi(\C^n)$ is empty.

\begin{proposition} \label{condition for Phi analytic wavefront to be empty}
Let $\Phi$ be a strictly plurisubharmonic quadratic form on $\C^n$. Suppose $u \in H^{-\infty}_\Phi(\C^n)$. Then $\textrm{WF}^{1/2}_\Phi(u) = \emptyset$ if and only if there exist $C,c>0$ such that
\begin{align*}
    \abs{u(z)} \le C e^{\Phi(z)-c\abs{z}^2}, \ \ \ z \in \C^n.
\end{align*}
\end{proposition}

Combining Proposition \ref{empty analytic wavefront set} with Proposition \ref{condition for Phi analytic wavefront to be empty} gives

\begin{corollary}
Suppose $u \in \mathcal{S}'(\R^n)$. Then $u$ is a Gelfand-Shilov test function if and only if there is an FBI phase $\varphi$ with associated weight $\Phi$ and constants $C, c>0$ such that
\begin{align} \label{FBI side Gelfand-Shilov estimate}
    \abs{\mathcal{T}_\varphi u(z)} \le C e^{\Phi(z)-c \abs{z}^2}, \ \ \ z \in \C^n.
\end{align}
If $u \in \mathcal{S}'(\R^n)$ and $\mathcal{T}_\varphi u$ satisfies (\ref{FBI side Gelfand-Shilov estimate}), then for every other choice of FBI phase $\varphi_1$ with associated weight $\Phi_1$, the function $\mathcal{T}_{\varphi_1} u$ satisfies the estimate (\ref{FBI side Gelfand-Shilov estimate}) for a potentially different choice of constants $C, c>0$ and with $\Phi$ replaced by $\Phi_1$.
\end{corollary}


\section{Bergman Representations of Metaplectic Fourier Integral Operators}

In this section we consider metaplectic Fourier integral operators and their action on exponentially weighted spaces of entire functions. We begin with a formal discussion. A metaplectic Fourier integral operator is an operator of the form
\begin{align} \label{Formal FIO}
    G u(z) = a \iint e^{i \phi(z,w,\theta)} u(w) \, dw \wedge d\theta, \ \ z \in \C^n,
\end{align}
where $a \in \C$ is a constant, $\phi(z,w,\theta)$ is a holomorphic quadratic form on $\C^{2n+N} = \C^n_z \times \C^n_w \times \C^N_\theta$. We assume that $\phi$ is a non-degenerate phase function in the sense of H\"{o}rmander \cite{FIO1}, i.e.
\begin{align} \label{nondegeneracy_condition}
    d \p_{\theta_1} \phi, . . . , d \p_{\theta_N} \phi \ \textrm{are linearly independent over $\C$}.
\end{align}
Let
\begin{align*}
    C_\phi = \set{(z,w,\theta) \in \C^{2n+N}}{\p_\theta \phi(z,w,\theta)=0}
\end{align*}
be the \emph{critical set of $\phi$}. Since $\p_\theta \phi$ is a $\C$-linear function of $(z,w,\theta) \in \C^{2n+N}$, the critical set of $\phi$ is a linear subspace of $\C^{2n+N}$, and the non-degeneracy of $\phi$ implies that $\dim_\C C_\phi = 2n$. We associate to $G$ the complex canonical relation $\kappa \subset \C^{2n} \times \C^{2n}$ given implicitly by
\begin{align} \label{canonical transformation associated to metaplectic FIO}
    \kappa : (w, -\p_w \phi (z,w,\theta)) \mapsto (z, \p_z \phi (z,w,\theta)), \ \ \ (z,w,\theta) \in C_\phi.
\end{align}
In the sequel, we shall always assume that $\kappa$ is the graph of a complex linear canonical transformation $\C^{2n} \rightarrow \C^{2n}$. In this situation, we say that $G$ \emph{quantizes $\kappa$} or that \emph{$\kappa$ is the underlying canonical transformation of $G$.}

We will now discuss how the formal Fourier integral operator (\ref{Formal FIO}) may be realized as a bounded linear operator between exponentially weighted spaces of entire functions by making an appropriate choice of the contour of integration. Let $\Phi_1$ and $\Phi_2$ be strictly plurisubharmonic quadratic forms on $\C^n$ and let $H_{\Phi_1}(\C^n)$ and $H_{\Phi_2}(\C^n)$ be their associated Bargmann spaces. Suppose that
\begin{align} \label{mapping of I lagrangians for FIO}
    \kappa(\Lambda_{\Phi_{2}}) = \Lambda_{\Phi_{1}},
\end{align}
where $\Lambda_{\Phi_1}$ and $\Lambda_{\Phi_2}$ are as in (\ref{IR subspace associated to Phi}) with $\Phi$ replaced by $\Phi_1$ and $\Phi_2$, respectively. Following \cite{PTSymmetric} Appendix B, the plurisubharmonic quadratic form
\begin{align*}
    \C^n \times \C^N \ni (w,\theta) \mapsto -\textrm{Im} \ \phi(0,w,\theta) + \Phi_{2}(w)
\end{align*}
is non-degenerate of signature $(n+N, n+N)$. Then, following either Proposition B.3 of \cite{PTSymmetric} or the general theory of \cite{AnalyticMicrolocal_Analysis}, we may conclude that there exists a real, smooth, $(n+N)$-dimensional contour $\Gamma(z)$ in $\C^{n+N}$, depending smoothly on $z \in \C^n$, such that $Gu$, when equipped with $\Gamma(z)$,  is well-defined as an element of $H_{\Phi_{1}}(\C^n)$ for $u \in H_{\Phi_{2}}(\C^n)$ and that (\ref{Formal FIO}) defines a bounded linear transformation
\begin{align} \label{metaplectic_FIO_is_bounded}
    G: H_{\Phi_{2}}(\C^n) \rightarrow H_{\Phi_{1}}(\C^n).
\end{align}

Next, we recount some recent results from \cite{ComplexFIOs} concerning the Bergman representation of a metaplectic Fourier integral operator $G$ whose underlying complex canonical transformation $\kappa$ satisfies (\ref{mapping of I lagrangians for FIO}). The following proposition summarizes the main results that we shall need. Recall that if $\Phi(z)$ is a strictly plurisubharmonic quadratic form on $\C^n$, then the polarization $\Psi(z,\theta)$ of $\Phi(z)$ is the unique holomorphic quadratic form on $\C^{2n} = \C^n_z \times \C^n_\theta$ such that $\Psi(z,\overline{z}) = \Phi(z)$ for all $z \in \C^n$.

\begin{proposition} \label{main imported proposition}
Let $G$, $\Phi_1$, $\Phi_2$, and $\kappa$ be as above. Let
\begin{align} \label{definition of pr's}
    \textrm{pr}_{\Phi_j} = \pi_1|_{\Lambda_{\Phi_j}}, \ \ \ j=1,2,
\end{align}
and let $\kappa^\flat$ be the $\R$-linear isomorphism
\begin{align} \label{definition of kappa flat}
    \kappa^\flat = \textrm{pr}_{\Phi_1} \circ \kappa \circ (\textrm{pr}_{\Phi_2})^{-1}: \C^n \rightarrow \C^n.
\end{align}
Then $G$ may be written uniquely in the form
\begin{align} \label{definition of G}
    Gu(z) = \hat{a} \int_{\C^n} e^{2 \Psi(z,\overline{w})} u(w) e^{-2\Phi_2(w)} \, L(dw), \ \ \ u \in H_{\Phi_2}(\C^n),
\end{align}
where $\hat{a} \in \C$ is a constant and $\Psi(z,\theta)$ is a holomorphic quadratic form on $\C^{2n} = \C^n_z \times \C^n_\theta$, depending only on $\kappa$, $\Phi_1$, and $\Phi_2$, having the following properties:
\begin{enumerate}
    \item if $\Psi_{2}(z, \theta)$ denotes the polarization of $\Phi_2$, then
    \begin{align}
        \phi(z,w,\theta) = \frac{2}{i} \Psi(z, \theta) - \frac{2}{i} \Psi_2(w,\theta), \ \ \ (z,w,\theta) \in \C^{3n},
    \end{align}
    is a non-degenerate holomorphic phase function generating $\textrm{graph}(\kappa)$:
    \begin{align}
        \kappa: \left(w, \frac{2}{i} \p_w \Psi_{2}(w,\theta) \right) \mapsto \left(z, \frac{2}{i} \p_z \Psi(z,\theta) \right), \ \ \ \p_\theta \Psi(z,\theta) = \p_\theta \Psi_{2}(w,\theta);
    \end{align}
    \item The real part of $\Psi(z,\theta)$ satisfies
    \begin{align} \label{quadratic form identity}
        2 \textrm{Re} \ \Psi(z,\theta) = \Phi_{1}(z)+\Phi_{2}(\overline{\theta})-R(z, \theta), \ \ \ z,\theta \in \C^n,
    \end{align}
    where $R(z,\theta)$ a non-negative quadratic form on $\C^{2n} = \C_z^n \times \C_\theta^n$ such that
    \begin{align} \label{main big O identity}
        c \abs{z-\kappa^\flat(\overline{\theta})}^2 \le R(z,\theta) \le C \abs{z-\kappa^\flat(\overline{\theta})}^2, \ \ z,\theta \in \C^n,
    \end{align}
    for some $C,c>0$.
\end{enumerate}
\end{proposition}

If $G$ is a metaplectic Fourier integral operator satisfying (\ref{metaplectic_FIO_is_bounded}), then we refer to (\ref{definition of G}) as the \emph{Bergman form of $G$}.

\begin{example} \label{identity_example}
Consider the formal Fourier integral operator
\begin{align} \label{example_of_identity}
	Gu(z) = \frac{1}{(2\pi)^n} \iint e^{i(z-w) \cdot \theta} u(w) \, dw \wedge d\theta, \ \ z \in \C^n.
\end{align}
The phase function
\begin{align}
	\phi(z,w,\theta) = (z-w) \cdot \theta, \ \ (z,w, \theta) \in \C^{3n},
\end{align}
is easily seen to satisfy H\"{o}rmander's non-degeneracy condition (\ref{nondegeneracy_condition}). A direct computation shows that the complex linear canonical transformation $\kappa: \C^{2n} \rightarrow \C^{2n}$ generated by $\phi$ is the identity
\begin{align} \label{identity canonical transformation}
	\kappa(z,\zeta) = (z,\zeta), \ \ (z,\zeta) \in \C^{2n}.
\end{align}
If $\Phi$ is any strictly plurisubharmonic quadratic form on $\C^n$, the formal Fourier integral operator (\ref{example_of_identity}) may realized as a bounded linear transformation $H_\Phi(\C^n) \rightarrow H_\Phi(\C^n)$ by integrating over the contour
\begin{align} \label{good_contour_identity}
	\Gamma(z): w \mapsto \theta = \frac{2}{i} \p_z \Phi(z) + i C \overline{(z-w)}, \ \ w \in \C^n,
\end{align}
where $C\gg0$ is sufficiently large. By the `complex Fourier inversion theorem' (see the proof of Proposition 1.3.4 in \cite{Minicourse} or the proof of Theorem 13.6 in \cite{SemiclassicalAnalysis}),
\begin{align}
	Gu = u
\end{align}
for all $u \in H_\Phi(\C^n)$. After an appropriate contour deformation and $\C$-linear change of variables (again, see the proof of Proposition 1.3.4 in \cite{Minicourse} or the proof of Theorem 13.6 in \cite{SemiclassicalAnalysis}), the operator (\ref{example_of_identity}) may be rewritten as
\begin{align} \label{Bergman_form_of_identity}
	Gu(z) = C_\Phi \int_{\C^n} e^{2\Psi(z,\overline{w})} u(w) e^{-2\Phi(w)} \, L(dw), \ \ u \in H_\Phi(\C^n),
\end{align}
where $\Psi(\cdot, \cdot)$ is the polarization of $\Phi$ and $C_\Phi = (2/\pi)^n \det{ \p^2_{z\overline{z}} \Phi}$. In other words, the Bergman form of $G$ that is guaranteed to exist by Proposition \ref{main imported proposition} is precisely the Bergman projector (\ref{definition_Bergman_projector}) associated to the weight $\Phi$. One may rewrite (\ref{Bergman_form_of_identity}) in the form
\begin{align}
	Gu(z) = \widetilde{C}_\Phi \iint_\Gamma e^{2 \Psi(z,\theta)-2 \Psi(w,\theta)} u(w) \, dw \wedge d\theta, \ \ u \in H_\Phi(\C^n),
\end{align}
where $\widetilde{C}_\Phi = (i/2)^n C_\Phi$ and the contour of integration is the anti-diagonal
\begin{align}
	\Gamma = \set{(w,\theta) \in \C^{2n}}{\theta = \overline{w}}.
\end{align}
The strict plurisubharmonicity of $\Phi$ implies that the phase function
\begin{align} \label{Bergman phase function}
	\frac{2}{i} \Psi(z,\theta)-\frac{2}{i} \Psi(w,\theta), \ \ (z,w,\theta) \in \C^{3n},
\end{align}
satisfies the non-degeneracy condition (\ref{nondegeneracy_condition}), and one may easily verify that (\ref{Bergman phase function}) generates the identity map (\ref{identity canonical transformation}).

\end{example}

Using Proposition \ref{main imported proposition}, we can give a simple proof that metaplectic Fourier integral operators extend uniquely to bounded linear transformations on $H^s_{\Phi}(\C^n)$ for every $s \in \R$.

\begin{proposition} \label{mapping properties of metaplectic FIO}
	Let $\Phi_1$ and $\Phi_2$ be strictly plurisubharmonic quadratic forms on $\C^n$ and let $\kappa: \C^{2n} \rightarrow \C^{2n}$ be a complex linear canonical transformation such that $\kappa(\Lambda_{\Phi_{2}}) = \Lambda_{\Phi_{1}}$. If $G: H_{\Phi_{2}}(\C^n) \rightarrow H_{\Phi_{1}}(\C^n)$ is a metaplectic Fourier integral operator quantizing $\kappa$, then $G$ extends uniquely to a bounded linear transformation
	\begin{align}
		G: H^s_{\Phi_{2}}(\C^n) \rightarrow H^s_{\Phi_{1}}(\C^n)
	\end{align}
	for every $s \in \R$. Consequently, $G$ restricts to a continuous linear transformation $H^\infty_{\Phi_{2}}(\C^n) \rightarrow H^\infty_{\Phi_1}(\C^n)$ and extends uniquely to a continuous linear transformation $H^{-\infty}_{\Phi_{2}}(\C^n) \rightarrow H^{-\infty}_{\Phi_{1}}(\C^n)$.
\end{proposition}
\begin{proof}
Let $\kappa^\flat: \C^n \rightarrow \C^n$ be as in (\ref{definition of kappa flat}). By Proposition \ref{main imported proposition}, we may write $G$ uniquely in Bergman form as
\begin{align} \label{operator_I_am_going_to_reference}
	Gu(z) = \hat{a} \int_{\C^n} e^{2 \Psi(z, \overline{w})} u(w) e^{-2\Phi_{2}(w)} \, L(dw), \ \ u \in H_{\Phi_{2}}(\C^n),
\end{align}
where $\hat{a} \in \C$ and $\Psi(\cdot, \cdot)$ is a holomorphic quadratic form on $\C^{2n}$ such that
\begin{align} \label{Bergman_phase_estimate}
	2 \textrm{Re} \ \Psi(z,\overline{w}) - \Phi_{1}(z) - \Phi_{2}(w) \le -c \abs{z- \kappa^\flat(w)}^2, \ \ z , w \in \C^n,
\end{align}
for some $c>0$. Now (\ref{operator_I_am_going_to_reference}) may be rewritten as
\begin{align}
	Gu(z) = \int_{\C^n} K(z,w) u(w) \, L(dw), \ \ u \in H_{\Phi_{2}}(\C^n),
\end{align}
where
\begin{align}
	K(z,w) = \hat{a} e^{2 \Psi(z,\overline{w}) - 2 \Phi_{2}(w)}, \ \ z,w \in \C^n.
\end{align}
Let $s \in \R$ be arbitrary. To see that
\begin{align}
	G = \mathcal{O}(1): H^s_{\Phi_{2}}(\C^n) \rightarrow H^s_{\Phi_{1}}(\C^n),
\end{align}
we consider the reduced kernel
\begin{align}
	K_{\textrm{red}}(z,w) = \langle z \rangle^{s} e^{-\Phi_{1}(z)} K(z,w) \langle w \rangle^{-s} e^{\Phi_{2}(w)}, \ \ z, w \in \C^n.
\end{align}
Since $\kappa^\flat: \C^n \rightarrow \C^n$ is an invertible linear transformation, there are $C,c>0$ such that
\begin{align} \label{equivalence with canonical transformation}
	c \langle \kappa^\flat(w) \rangle \le \langle w \rangle \le C \langle \kappa^\flat(w) \rangle
\end{align}
for all $w \in \C^n$. From (\ref{Bergman_phase_estimate}) and (\ref{equivalence with canonical transformation}), we get that are $C,c>0$ such that
\begin{align}
	\abs{K_{\textrm{red}} (z,w)} \le C \langle z - \kappa^\flat(w) \rangle^{\abs{s}} e^{-c \abs{z-\kappa^\flat(w)}^2}, \ \ z,w \in \C^n.
\end{align}
Because
\begin{align}
	\sup_{z \in \C^n} \int_{\C^n} \abs{K_{\textrm{red}}(z,w)} \, L(dw) < \infty \ \ \textrm{and} \ \ \sup_{w \in \C^n} \int_{\C^n} \abs{K_{\textrm{red}}(z,w)} \, L(dz) < \infty,
\end{align}
Schur's lemma implies that the operator (\ref{operator_I_am_going_to_reference}) is $\mathcal{O}(1): H^s_{\Phi_{2}}(\C^n) \rightarrow H^s_{\Phi_{1}}(\C^n)$. As $H^\infty_{\Phi_{2}}(\C^n)$ is dense in $H^s_{\Phi_{2}}(\C^n)$ by Proposition \ref{density_proposition}, we conclude that $G$ extends uniquely to a bounded linear transformation $H^s_{\Phi_{2}}(\C^n) \rightarrow H^s_{\Phi_{1}}(\C^n)$. It follows immediately that $G$ restricts to a continuous linear transformation $H^\infty_{\Phi_{2}}(\C^n) \rightarrow H^\infty_{\Phi_{1}}(\C^n)$. To extend $G$ to $H^{-\infty}_{\Phi_{2}}(\C^n)$, we define $Gu$ for $u \in H^{-\infty}_{\Phi_{2}}(\C^n)$ using the formula (\ref{operator_I_am_going_to_reference}). To prove the continuity of $G: H^{-\infty}_{\Phi_{2}}(\C^n) \rightarrow H^{-\infty}_{\Phi_{1}}(\C^n)$ it suffices to show that for every $s \in \R$ the restriction of $G$ to $H^s_{\Phi_{2}}(\C^n)$ is continuous $H^s_{\Phi_{2}}(\C^n) \rightarrow H^{-\infty}_{\Phi_{1}}(\C^n)$ (see \cite{Conway} Chapter IV Proposition 5.7). But this is immediate since we have already established the continuity of $G: H^s_{\Phi_{2}}(\C^n) \rightarrow H^s_{\Phi_{1}}(\C^n)$ and the inclusion $H^s_{\Phi_{1}}(\C^n) \hookrightarrow H^{-\infty}_{\Phi_{1}}(\C^n)$ is continuous by definition of the topology on $H^{-\infty}_{\Phi_{1}}(\C^n)$. Since $H^{\infty}_{\Phi_{2}}(\C^n)$ is dense in $H^{-\infty}_{\Phi_{2}}(\C^n)$ by Proposition \ref{density_proposition}, it follows that $G$ extends uniquely to a continuous linear transformation $H^{-\infty}_{\Phi_{2}}(\C^n) \rightarrow H^{-\infty}_{\Phi_{1}}(\C^n)$.
\end{proof}


\section{Metaplectic Fourier Integral Operators and Propagation of $1/2$-Gelfand-Shilov Singularities}

Let $\C^{2n}$ be equipped with the standard complex symplectic form $\sigma$ and let $\Sigma$ be a maximally totally real subspace of $\C^{2n}$. Let $\iota_\Sigma$ be the unique antilinear involution of $\C^{2n}$ fixing $\Sigma$. Following the terminology of \cite{Minicourse}, we say that a $\C$-Lagrangian subspace $\Lambda$ of $\C^{2n}$ is \emph{positive relative to $\Sigma$} if
\begin{align} \label{defintion of positivity}
    \frac{1}{i} \sigma(X, \iota_\Sigma X) \ge 0, \ \ \ X \in \Sigma.
\end{align}
If equality holds in (\ref{defintion of positivity}) only when $X=0$, we say that $\Lambda$ is \emph{strictly positive relative to $\Sigma$}. 

One may extend the notion of positivity to complex linear canonical transformations of $\C^{2n}$. If $\kappa: \C^{2n} \rightarrow \C^{2n}$ is a complex linear canonical transformation and $\Sigma_1, \Sigma_2 \subset \C^{2n}$ are maximally totally real subspaces of $\C^{2n}$ with associated antilinear involutions $\iota_{\Sigma_1}$ and $\iota_{\Sigma_2}$, respectively, then we say that $\kappa$ is {\emph{positive relative to $(\Sigma_1, \Sigma_2)$}} if
\begin{align} \label{definition positivity canonical transformation}
    \frac{1}{i}\left(\sigma(\kappa(X), \iota_{\Sigma_1} \kappa(X))-\sigma(X, \iota_{\Sigma_2} X) \right) \ge 0, \ \ \ X \in \C^{2n}.
\end{align}
If the inequality in (\ref{definition positivity canonical transformation}) is strict for all $X \neq 0$, then $\kappa$ is said to be {\emph{strictly positive relative to $(\Sigma_1, \Sigma_2)$}}. In the case when $\kappa$ is positive, resp. strictly positive, relative to $(\Sigma_1, \Sigma_2)$ and $\Sigma_1 = \Sigma_2 = \Sigma$, then we simply say that \emph{$\kappa$ is positive, resp. strictly positive, relative to $\Sigma$.}

Let $\Phi_1$ and $\Phi_2$ be strictly plurisubharmonic quadratic forms on $\C^n$. In \cite{ComplexFIOs}, it was shown that a complex linear canonical transformation $\kappa: \C^{2n} \rightarrow \C^{2n}$ is positive relative to $(\Lambda_{\Phi_1}, \Lambda_{\Phi_2})$ if and only if
\begin{align}
    \kappa(\Lambda_{\Phi_2}) = \Lambda_{\Phi}
\end{align}
where $\Phi$ is a strictly plurisubharmonic quadratic form on $\C^{2n}$ such that $\Phi \le \Phi_1$. In particular, if $\Phi$ is a strictly plurisubharmonic quadratic form on $\C^n$ and $\kappa$ is positive relative to $\Lambda_{\Phi}$, then
\begin{align} \label{positive relative to a weight image characterization}
    \kappa(\Lambda_{\Phi}) = \Lambda_{\tilde{\Phi}}
\end{align}
for some strictly plurisubharmonic quadratic form $\tilde{\Phi}$ on $\C^n$ such that $\tilde{\Phi} \le \Phi$. In this case, there is a very useful characterization of the $I$-isotropic subspace $\Lambda_\Phi \cap \kappa(\Lambda_\Phi)$ in terms of $\Phi$ and $\tilde{\Phi}$. Namely, if (\ref{positive relative to a weight image characterization}) holds, then
\begin{align} \label{characterization of clean intersection}
    \pi_1(\Lambda_\Phi \cap \kappa(\Lambda_\Phi)) = \textrm{Rad}(\Phi-\tilde{\Phi}).
\end{align}
Indeed, since $\Phi-\tilde{\Phi}$ is a non-negative quadratic form, we have
\begin{align}
    \Phi(z)-\tilde{\Phi}(z) = 0 \iff \nabla_{\textrm{Re} \ z, \textrm{Im} \ z}(\Phi - \tilde{\Phi})(z) = 0
    &\iff \p_z (\Phi - \tilde{\Phi})(z) = 0.
\end{align}
Hence
\begin{align}
    \left(z, \frac{2}{i} \p_z \Phi(z) \right) = \left(z, \frac{2}{i} \p_z \tilde{\Phi}(z) \right) \iff z \in \textrm{Rad}(\Phi-\tilde{\Phi}).
\end{align}

Suppose $\kappa: \C^{2n} \rightarrow \C^{2n}$ is a complex linear canonical transformation that is positive relative to $\Lambda_\Phi$ and let $\tilde{\Phi}$ be as in (\ref{positive relative to a weight image characterization}). If $G$ is a metaplectic Fourier integral operator quantizing $\kappa$, then $G$ is a continuous linear transformation $H^{-\infty}_{\Phi}(\C^n) \rightarrow H^{-\infty}_{\tilde{\Phi}}(\C^n)$ by Proposition \ref{mapping properties of metaplectic FIO}. Since also $H^{-\infty}_{\tilde{\Phi}}(\C^n) \hookrightarrow H^{-\infty}_{\Phi}(\C^n)$ continuously, we may regard $G$ as a continuous linear transformation from $H^{-\infty}_{\Phi}(\C^n)$ to itself. Consequently, $\textrm{WF}^{1/2}_\Phi(Gu)$ is well-defined for any $u \in H^{-\infty}_\Phi(\C^n)$. 

We wish to explore the relationship between $\textrm{WF}^{1/2}_\Phi(u)$ and $\textrm{WF}^{1/2}_\Phi(Gu)$ when $u \in H^{-\infty}_\Phi(\C^n)$ and $G$ is a metaplectic Fourier integral operator whose underlying canonical transformation $\kappa$ is positive relative to $\Phi$. The next theorem shows that $G$ regularizes any $1/2$-Gelfand-Shilov singularities of $u$ that are outside of $\textrm{Rad}(\Phi-\tilde{\Phi})$ and transports those that lie within $\textrm{Rad}(\Phi-\tilde{\Phi})$ by $\kappa^\flat$, where $\kappa^\flat$ is as in (\ref{definition of kappa flat}).

\begin{theorem} \label{main theorem}
Let $\Phi$ be a strictly plurisubharmonic quadratic form on $\C^n$, let $\kappa: \C^{2n} \rightarrow \C^{2n}$ be a complex linear canonical transformation that is positive relative to $\Phi$, let $\tilde{\Phi}$ be as in (\ref{positive relative to a weight image characterization}), and let $\kappa^\flat: \C^n \rightarrow \C^n$ be the $\R$-linear isomorphism
\begin{align*}
    \kappa^\flat = \textrm{pr}_{\tilde{\Phi}} \circ \kappa \circ \textrm{pr}_{\Phi}^{-1} : \C^n \rightarrow \C^n,
\end{align*}
where $\textrm{pr}_\Phi$ and $\textrm{pr}_{\tilde{\Phi}}$ are the restrictions of $\pi_1$ to $\Lambda_\Phi$ and $\Lambda_{\tilde{\Phi}}$ respectively. If $G$ is a metaplectic Fourier integral operator quantizing $\kappa$, realized as a continuous linear transformation from $H^{-\infty}_\Phi(\C^n)$ to itself, then, for any $u \in H^{-\infty}_\Phi(\C^n)$, we have
\begin{align} \label{main wavefront propagation inclusion}
    \textrm{WF}^{1/2}_\Phi(Gu) \subset \kappa^\flat(\textrm{WF}^{1/2}_\Phi(u)) \cap \textrm{Rad}(\Phi-\tilde{\Phi}).
\end{align}
If, in addition, $G$ is non-zero and $\kappa(\Lambda_\Phi \cap \Lambda_{\tilde{\Phi}}) = \Lambda_\Phi \cap \Lambda_{\tilde{\Phi}}$, then
\begin{align} \label{main wavefront equality}
    \textrm{WF}^{1/2}_\Phi(Gu) = \kappa^\flat(\textrm{WF}^{1/2}_\Phi(u)) \cap \textrm{Rad}(\Phi-\tilde{\Phi})
\end{align}
for every $u \in H^{-\infty}_\Phi(\C^n)$.
\end{theorem}

We begin the proof of Theorem \ref{main theorem} by establishing the inclusion (\ref{main wavefront propagation inclusion}). It suffices to show
\begin{align} \label{first flipped inclusion}
    \C^n \backslash \textrm{Rad}(\Phi - \tilde{\Phi}) \subset \C^n \backslash \textrm{WF}^{1/2}_\Phi(Gu).
\end{align}
and
\begin{align} \label{second flipped inclusion}
    \kappa^\flat(\C^n \backslash \textrm{WF}^{1/2}_\Phi(u)) \subset \C^n \backslash \textrm{WF}^{1/2}_\Phi(Gu).
\end{align}
Let
\begin{align*}
    Gu(z) = \hat{a} \int_{\C^n} e^{2\Psi(z,\overline{w})} u(w) e^{-2\Phi(w)} \, L(dw), \ \ \ u \in H^{-\infty}_\Phi(\C^n),
\end{align*}
where $\hat{a} \neq 0$, be the Bergman form of $G$ given in Proposition \ref{main imported proposition}. Suppose that $z_0 \in \C^n \backslash \textrm{Rad}(\Phi-\tilde{\Phi})$. There is an open conic neighborhood $V$ of $z_0$ in $\C^n \backslash \{0\}$ such that
\begin{align*}
    \Phi(z)-\tilde{\Phi}(z) \ge c \abs{z}^2
\end{align*}
for all $z \in V$. In view of (\ref{quadratic form identity}) and (\ref{main big O identity}), for all $z \in V$, we have
\begin{align*}
   \abs{Gu(z)} e^{-\Phi(z)} &\le C e^{-c \abs{z}^2} \int_{\C^n} e^{-\Phi(w)-c \abs{z-\kappa^\flat(w)}^2} \abs{u(w)} \, L(dw) \\ &\le C \norm{u}_{H^s_\Phi(\C^n)} e^{-c \abs{z}^2} \left( \int_{\C^n} \langle w \rangle^{-2s} e^{-c \abs{z-\kappa^\flat(w)}^2} \, L(dw) \right)^{1/2},
\end{align*}
where $s \in \R$ is such that $u \in H^s_\Phi(\C^n)$. Since $\kappa^\flat$ is a $\R$-linear isomorphism $\C^n \rightarrow \C^n$, we see that
\begin{align*}
    \left(\int_{\C^n} \langle w \rangle^{-2s} e^{-c \abs{z - \kappa^\flat(w)}^2} \, L(dw) \right)^{1/2} \le C \langle z \rangle^s.
\end{align*}
It follows
\begin{align*}
    \abs{Gu(z)} e^{-\Phi(z)} \le C e^{-c \abs{z}^2}, \ \ \ z \in V.
\end{align*}
Hence (\ref{first flipped inclusion}) holds.

Let $z_0 \in \kappa^\flat(\C^n \backslash \textrm{WF}^{1/2}_\Phi(u))$. If $z_0 = 0$, then trivially $z_0 \in \C^n \backslash \textrm{WF}^{1/2}_\Phi(Gu)$. If $z_0 \neq 0$, write $z_0 = \kappa^\flat(w_0)$ for some unique $w_0 \in \C^n \backslash \textrm{WF}^{1/2}_\Phi(u)$ and let $V$ be an open conic neighborhood of $w_0$ in $\C^n \backslash \{0\}$ such that
\begin{align} \label{estimate_in_V}
    \abs{u(w)} \le C e^{\Phi(w)-c \abs{w}^2}, \ \ \ w \in V.
\end{align}
Let $\tilde{V}$ be an open conic neighborhood of $z_0$ in $\C^n \backslash \{0\}$ such that $\tilde{V} \subset \subset \kappa^\flat(V)$. From Proposition \ref{main imported proposition} and the fact that $\tilde{\Phi} \le \Phi$, we get
\begin{align}
    \abs{Gu (z) e^{-\Phi(z)}} \le C \left(\int_{V} + \int_{\C^n \backslash V} \right) e^{-c \abs{z-\kappa^\flat (w)}^2} \abs{u(w)} e^{-\Phi(w)} \, L(dw) =: I(z) + II(z), \ \ \ z \in \tilde{V}.
\end{align}
In view of (\ref{estimate_in_V}),
\begin{align*}
    I(z) \le C \int_{V} e^{-c \abs{(\kappa^\flat)^{-1}(z)- w}^2 - c \abs{w}^2} \, L(dw) \le C e^{-c \abs{z}^2}, \ \ \ z \in \C^n.
\end{align*}
To estimate $II(z)$, we observe that the quadratic form
\begin{align*}
    (z,w) \mapsto \abs{z-\kappa^\flat (w)}^2
\end{align*}
is non-vanishing for $(z,w) \in \tilde{V} \times (\C^n\backslash \{0\}) \backslash V$. By homogeneity, there is a constant $c>0$ such that
\begin{align*}
    \abs{z-\kappa^\flat (w)}^2 \ge c(\abs{z}^2+\abs{w}^2)
\end{align*}
for all $z \in \tilde{V}$ and $w \in \C^n \backslash V$. As a result,
\begin{align*}
    II(z) \le C \norm{u}_{H^s_\Phi(\C^n)} e^{-c \abs{z}^2} \left(\int_{\C^n} \langle w \rangle^{-2s} e^{-c \abs{w}^2} \, L(dw)\right)^{1/2} \le C e^{-c \abs{z}^2}
\end{align*}
for all $z \in \tilde{V}$. This establishes that
\begin{align*}
    \abs{Gu(z)} e^{-\Phi(z)} \le C e^{-c \abs{z}^2}, \ \ \ z \in \tilde{V}.
\end{align*}
Therefore (\ref{second flipped inclusion}) holds. We conclude that the inclusion (\ref{main wavefront propagation inclusion}) is true.


To establish the equality (\ref{main wavefront equality}) under the additional assumption that $\kappa(\Lambda_\Phi \cap \Lambda_{\tilde{\Phi}}) = \Lambda_\Phi \cap \Lambda_{\tilde{\Phi}}$, we first prove the following lemma.

\begin{lemma} \label{main lemma to prove exact propagation}
Let $\Phi$, $\tilde{\Phi}$, $\kappa$, and $\kappa^\flat$ be as in Theorem \ref{main theorem} and assume $\kappa(\Lambda_\Phi \cap \Lambda_{\tilde{\Phi}}) = \Lambda_\Phi \cap \Lambda_{\tilde{\Phi}}$. Suppose that $\widetilde{G}$ is a metaplectic Fourier integral operator quantizing $\kappa^{-1}$, realized as a continuous linear transformation $H^{-\infty}_{\tilde{\Phi}}(\C^n) \rightarrow H^{-\infty}_{\Phi}(\C^n)$. Then, for any $v \in H^{-\infty}_{\tilde{\Phi}}(\C^n)$,
\begin{align}
    \textrm{WF}^{1/2}_\Phi(\widetilde{G} v) \cap \textrm{Rad}(\Phi-{\tilde{\Phi}}) \subset (\kappa^\flat)^{-1}(\textrm{WF}^{1/2}_\Phi(v)) \cap \textrm{Rad}(\Phi-\tilde{\Phi}).
\end{align}
\end{lemma}
\begin{proof}
Let $v \in H^{-\infty}_{\tilde{\Phi}}(\C^n)$ be arbitrary. It suffices to show that
\begin{align} \label{complementary inclusion in important lemma}
    \textrm{Rad}(\Phi - \tilde{\Phi}) \backslash (\kappa^\flat)^{-1}\left(\textrm{WF}^{1/2}_\Phi(v)\right) \subset \textrm{Rad}(\Phi-\tilde{\Phi}) \backslash  \textrm{WF}^{1/2}_\Phi(\widetilde{G}v).
\end{align}
Because $(\kappa^\flat)^{-1}: \C^n \rightarrow \C^n$ is invertible,
\begin{align}
\begin{split}
    \textrm{Rad}(\Phi - \tilde{\Phi}) \backslash (\kappa^\flat)^{-1}\left(\textrm{WF}^{1/2}_\Phi(v)\right) &= \textrm{Rad}(\Phi- \tilde{\Phi}) \cap \left[\C^n \backslash (\kappa^\flat)^{-1}(\textrm{WF}^{1/2}_\Phi(v))\right] \\
    & = \textrm{Rad}(\Phi-\tilde{\Phi}) \cap (\kappa^\flat)^{-1} \left(\C^n \backslash \textrm{WF}^{1/2}_\Phi(v)\right).
\end{split}
\end{align}
Thus (\ref{complementary inclusion in important lemma}) is equivalent to
\begin{align} \label{equivalent inclusion}
    \textrm{Rad}(\Phi-\tilde{\Phi}) \cap (\kappa^\flat)^{-1}\left(\C^n \backslash \textrm{WF}^{1/2}_\Phi(v) \right)  \subset \textrm{Rad}(\Phi-\tilde{\Phi}) \cap \C^n \backslash \textrm{WF}^{1/2}_\Phi(\widetilde{G}v).
\end{align}

Let $z_0 \in \textrm{Rad}(\Phi-\tilde{\Phi}) \cap (\kappa^\flat)^{-1}\left(\C^n \backslash \textrm{WF}^{1/2}_\Phi(v) \right)$. If $z_0 = 0$, then trivially $z_0 \in \textrm{Rad}(\Phi-\tilde{\Phi}) \cap \C^n \backslash \textrm{WF}^{1/2}_\Phi(\widetilde{G}v)$. If $z_0 \neq 0$, then we may write $z_0 =(\kappa^\flat)^{-1} w_0$ for some non-zero $w_0 \in \C^n \backslash \textrm{WF}^{1/2}_\Phi(v)$. In view of (\ref{characterization of clean intersection}) and our assumption that $\kappa(\Lambda_\Phi \cap \Lambda_{\tilde{\Phi}}) = \Lambda_\Phi \cap \Lambda_{\tilde{\Phi}}$, we have
\begin{align} \label{invariance of the radical}
  \kappa^\flat(\textrm{Rad}(\Phi-\tilde{\Phi})) = \pi_1(\kappa(\Lambda_\Phi \cap \Lambda_{\tilde{\Phi}})) = \pi_1(\Lambda_\Phi \cap \Lambda_{\tilde{\Phi}}) = \textrm{Rad}(\Phi-\tilde{\Phi}).
\end{align}
It follows that $w_0 \in \textrm{Rad}(\Phi-\tilde{\Phi}) \cap \C^n \backslash \textrm{WF}^{1/2}_\Phi(v)$. Let $\tilde{V} \subset \C^n \backslash \{0\}$ be an open conic neighborhood of $w_0$ in $\C^n \backslash \{0\}$ such that
\begin{align} \label{estimate for v(w)}
    \abs{v(w)} e^{-\Phi(w)} \le C e^{-\delta \abs{w}^2}, \ \ \ w \in \tilde{V},
\end{align}
for some $C, \delta>0$. Since $\Phi(w_0) - \tilde{\Phi}(w_0) = 0$, we may ensure, by taking $\tilde{V}$ smaller if necessary, that
\begin{align} \label{good estimate within Gamma}
    \Phi(w)-\tilde{\Phi}(w) \le \frac{1}{2} \delta \abs{w}^2
\end{align}
holds for all $w \in \tilde{V}$. Let $V$ be an open conic neighborhood of $z_0$ in $\C^n \backslash \{0\}$ such that 
\begin{align} \label{V is compactly contained}
V \subset \subset (\kappa^\flat)^{-1} (\tilde{V}).
\end{align}
We claim that
\begin{align} \label{step to conclude}
    \abs{\widetilde{G}v(z)}e^{-\Phi(z)} \le C e^{-c \abs{z}^2}
\end{align}
for all $z \in V$. Let
\begin{align*}
    \widetilde{G}u(z) = \tilde{a} \int_{\C^n} e^{2 \tilde{\Psi}(z,\overline{w})} u(w) e^{-2\tilde{\Phi}(w)} \, L(dw), \ \ \ u \in H^{-\infty}_{\tilde{\Phi}}(\C^n),
\end{align*}
where $\tilde{a} \neq 0$, be the Bergman form of $\widetilde{G}$. Since there is $c>0$ such that
\begin{align}
2 \textrm{Re} \ \tilde{\Psi}(z,\overline{w}) - \Phi(z)-\tilde{\Phi}(w) \le -c \abs{z-(\kappa^\flat)^{-1}(w)}^2, \ \ z, w \in \C^n,
\end{align}
we get, for $z \in V$,
\begin{align} \label{where A and B are defined}
    \abs{\widetilde{G}v(z)e^{-\Phi(z)}} &\le C \left(\int_{\tilde{V}} + \int_{\C^n \backslash \tilde{V}} \right) e^{-c \abs{z-(\kappa^\flat)^{-1}w}^2} \abs{v(w)} e^{-\tilde{\Phi}(w)} \, L(dw)
    =: I(z) + II(z).
\end{align}
Because (\ref{estimate for v(w)}) and (\ref{good estimate within Gamma}) hold within $\tilde{V}$, 
\begin{align*}
    I(z) &= C \int_{\tilde{V}} e^{-c \abs{z-(\kappa^\flat)^{-1}w}^2} e^{\Phi(w)-\tilde{\Phi}(w)} e^{-\delta \abs{w}^2} \, L(dw) \\
    &\le C \int_{\C^n} e^{-c \abs{z-(\kappa^\flat)^{-1}w}^2} e^{-\frac{\delta}{2} \abs{w}^2} \, L(dw) \\
    &\le C e^{-c \abs{z}^2}, \ \ \ z \in \C^n.
\end{align*}
To estimate $II(z)$, we notice that, thanks to (\ref{V is compactly contained}), the quadratic form
\begin{align*}
    (z,w) \mapsto \abs{z-(\kappa^\flat)^{-1}w}^2
\end{align*}
is non-vanishing for $(z,w) \in V \times (\C^n\backslash \{0\}) \backslash \tilde{V}$. By homogeneity, there is a constant $\gamma>0$ such that
\begin{align*}
    \abs{z-(\kappa^\flat)^{-1} w}^2 \ge \gamma (\abs{z}^2+\abs{w}^2)
\end{align*}
for all $(z,w) \in V \times \C^n \backslash \tilde{V}$. Letting $s \in \R$ be such that $v \in H^s_{\tilde{\Phi}}(\C^n)$, we obtain
\begin{align*}
    II(z) \le C \left(\int_{\C^n} e^{-2\gamma (\abs{z}^2+\abs{w}^2)} \langle w \rangle^{-2s} \, L(dw) \right)^{1/2} \le C e^{- c\abs{z}^2}, \ \ \ z \in V.
\end{align*}
This proves (\ref{step to conclude}). Therefore $z_0 \in \textrm{Rad}(\Phi-\tilde{\Phi}) \cap \C^n \backslash \textrm{WF}^{1/2}_\Phi(\widetilde{G}v)$. The lemma is proved.
\end{proof}

Now we can complete the proof of Theorem \ref{main theorem}. Assume that $\kappa(\Lambda_\Phi \cap \Lambda_{\tilde{\Phi}}) = \Lambda_\Phi \cap \Lambda_{\tilde{\Phi}}$. Let $\widetilde{G}$ be a non-zero metaplectic Fourier integral operator quantizing $\kappa^{-1}$, realized as a continuous linear transformation $H^{-\infty}_{\tilde{\Phi}}(\C^n) \rightarrow H^{-\infty}_\Phi(\C^n)$. By Proposition B.4 of \cite{PTSymmetric}, the operator $\widetilde{G} G : H^{-\infty}_\Phi(\C^n) \rightarrow H^{-\infty}_\Phi(\C^n)$ is a metaplectic Fourier integral operator quantizing the identity map on $\C^{2n}$. By Proposition \ref{main imported proposition} (see also Example \ref{identity_example}), the Bergman form of $\widetilde{G} G$ must be
\begin{align} \label{about to prove equal to Bergman projector}
    \widetilde{G} Gu = \hat{b} \int_{\C^n} e^{2\Psi(z,\overline{w})} u(w) e^{-2 \Phi(w)} \, L(dw), \ \ \ u \in H^{-\infty}_\Phi(\C^n),
\end{align}
where $\hat{b} \in \C$ and $\Psi(z,\theta)$ is the polarization of $\Phi$. By multiplying $\widetilde{G}$ by a non-zero constant if necessary, we may ensure that $\hat{b} = 2^n \pi^{-n} \det{\p^2_{z\overline{z}}\Phi}$. For this choice of $\hat{b}$, the righthand side of (\ref{about to prove equal to Bergman projector}) is the Bergman projector (\ref{definition_Bergman_projector}). From (\ref{Bergman_projection_is_reproducing}), we have
\begin{align} \label{composition is Bergman projector}
    \widetilde{G}G = I \ \ \textrm{on} \ \ H^{-\infty}_\Phi(\C^n).
\end{align}
Let $u \in H^{-\infty}_\Phi(\C^n)$ be given.
By Lemma \ref{main lemma to prove exact propagation} and (\ref{composition is Bergman projector}), we have
\begin{align} \label{penultimate line of the proof}
    \textrm{WF}^{1/2}_\Phi(u) \cap \textrm{Rad}(\Phi-\tilde{\Phi}) \subset (\kappa^{\flat})^{-1}(\textrm{WF}^{1/2}_\Phi(Gu)) \cap \textrm{Rad}(\Phi-\tilde{\Phi}).
\end{align}
By (\ref{main wavefront propagation inclusion}) and (\ref{invariance of the radical}),
\begin{align} \label{second line to end of proof}
    (\kappa^\flat)^{-1}(\textrm{WF}^{1/2}_\Phi(Gu)) \cap \textrm{Rad}(\Phi-\tilde{\Phi}) \subset \textrm{WF}^{1/2}_\Phi(u) \cap \textrm{Rad}(\Phi-\tilde{\Phi}).
\end{align}
Combining (\ref{penultimate line of the proof}) with (\ref{second line to end of proof}) and using (\ref{invariance of the radical}) gives
\begin{align}
    \textrm{WF}^{1/2}_\Phi(Gu) \cap \textrm{Rad}(\Phi-\tilde{\Phi}) = \kappa^\flat(\textrm{WF}^{1/2}_\Phi(u)) \cap \textrm{Rad}(\Phi-\tilde{\Phi}).
\end{align}
Since (\ref{main wavefront propagation inclusion}) also holds, we must have
\begin{align}
    \textrm{WF}^{1/2}_\Phi(Gu) = \kappa^\flat(\textrm{WF}^{1/2}_\Phi(u)) \cap \textrm{Rad}(\Phi-\tilde{\Phi}).
\end{align}
The proof of Theorem \ref{main theorem} is complete.


\section{The Bergman Representation of the Evolution Semigroup and the Proof of Theorem \ref{main theorem of the paper}}

Let $q=q(x,\xi)$ be a complex-valued quadratic form on $\R^{2n}$ with $\textrm{Re} \ q \ge 0$ and let $q^w(x,D)$ be the Weyl quantization of $q$. We consider the Schr\"{o}dinger initial value problem
\begin{align} \label{Schrodinger initial value problem}
    \begin{cases}
    \p_t u(t,x) + q^w(x,D)u(t,x) = 0, \ \ \ t \ge 0, \ x \in \R^n, \\
    u|_{t=0}= u_0 \in L^2(\R^n).
    \end{cases}
\end{align}
From the discussion on pages 425-426 of \cite{GeneralizedMehler}, we know that $q^w(x,D)$, regarded as an unbounded operator on $L^2(\R^n)$ equipped with its maximal domain
\begin{align}
	D_\textrm{max} = \set{u \in L^2(\R^n)}{q^w(x,D)u \in L^2(\R^n)},
\end{align}
generates a strongly continuous one-parameter semigroup $G(t) = e^{-tq^w(x,D)}$, $t \ge 0$, on $L^2(\R^n)$. We may regard $G(t)$ as the solution operator for the problem (\ref{Schrodinger initial value problem}).

Let $\varphi$ be an FBI phase function with associated FBI transform $\mathcal{T}_\varphi$ and strictly plurisubharmonic weight $\Phi(z)$. Let $\kappa_\varphi: \C^{2n} \rightarrow \C^{2n}$ be the complex canonical transformation generated by $\varphi$, and let $\tilde{q} = q \circ \kappa_\varphi^{-1}$. Applying $\mathcal{T}_\varphi$ to (\ref{Schrodinger initial value problem}) and using the complex Egorov theorem (\cite{SemiclassicalAnalysis} Theorem 13.9) gives
\begin{align}
	\begin{cases}
		\p_t \mathcal{T}_\varphi u(t,z) + \tilde{q}^w(z,D) \mathcal{T}_\varphi u(t,z) = 0, \ \ t \ge 0, \ z \in \C^n, \\
		\mathcal{T}_\varphi u|_{t=0} = \mathcal{T}_\varphi u_0 \in H_\Phi(\C^n).
	\end{cases}
\end{align}
Here
\begin{align} \label{definition complex Weyl quantization of q}
    \tilde{q}^w(z,D_z)u(z) = \frac{1}{(2\pi)^n} \iint_{\Gamma_\Phi(z)} e^{i(z-w) \cdot \theta} \tilde{q} \left(\frac{z+w}{2}, \theta \right) u(w) \, dw \wedge d\theta,
\end{align}
where
\begin{align} \label{contour of integration for complex Weyl quantization}
    \Gamma_{\Phi}(z) = \set{(w,\theta) \in \C^{2n}}{\theta = \frac{2}{i} \frac{\p \Phi}{\p z} \left(\frac{z+w}{2} \right)},
\end{align}
is the complex Weyl quantization of the symbol $\tilde{q}$. For further information regarding complex Weyl quantization, we refer the reader to Section 1.4 of \cite{Minicourse}, Chapter 13 of \cite{SemiclassicalAnalysis}, or Section 12.2 of \cite{Lectures_on_Resonances}. In particular, since $\norm{\p^\alpha \tilde{q}}_{L^\infty(\C^{2n})} < \infty$ for all $\abs{\alpha} \ge 2$, Proposition 12.6 of \cite{Lectures_on_Resonances} implies that $\tilde{q}^w(z,D) = \mathcal{O}(1): H^s_{\Phi}(\C^n) \rightarrow H^{s-2}_\Phi(\C^n)$ for every $s \in \R$. We also note that, since $\tilde{q}(z,\zeta)$ is a holomorphic quadratic form, the operator $\tilde{q}^w(z,D)$ acts as a quadratic differential operator on elements of $H^{-\infty}_\Phi(\C^n)$. Indeed, if
\begin{align} \label{differential_operator_FBI_side}
	\tilde{q}(z,\zeta) = \frac{1}{2} A_1 z \cdot z + A_2 z \cdot \zeta + \frac{1}{2} A_3 \zeta \cdot \zeta, \ \ (z,\zeta) \in \C^{2n},
\end{align}
where $A_1, A_2, A_3 \in M_{n \times n}(\C)$ with $A_1 = A_1^T$ and $A_3=A_3^T$, then
\begin{align} \label{quadratic_differential_operator_on_the_FBI_side_formula}
	\tilde{q}^w(z,D)u(z) = \left(\frac{1}{2}A_1 z \cdot z +A_2z \cdot D_z + \frac{1}{2i} \trace{A_2} + \frac{1}{2} A_3 D_z \cdot D_z \right)u(z), \ \ z \in \C^n,
\end{align}
where $D_z = \frac{1}{i} \p_z$, for all $u \in H^{-\infty}_\Phi(\C^n)$. 

We may view $\tilde{q}^w(z,D)$ as an unbounded operator on $H_\Phi(\C^n)$ equipped with the maximal domain
\begin{align}
	\widetilde{D}_\textrm{max} = \set{u\in H_\Phi(\C^n)}{\tilde{q}^w(z,D) u \in H_\Phi(\C^n)}.
\end{align}
As a consequence of the complex Egorov theorem, we have $\widetilde{D}_\textrm{max} = \mathcal{T}_\varphi \left(D_\textrm{max}\right)$. Since also $\mathcal{T}_\varphi: L^2(\R^n) \rightarrow H_{\Phi(\C^n)}$ is unitary, it follows that $\tilde{q}^w(z,D)$ generates a strongly continuous one-parameter semigroup $\widetilde{G}(t) = e^{-t \tilde{q}^w(z,D)}$, $t \ge 0$, on $H_\Phi(\C^n)$. This semigroup is related to $G(t)$ by
\begin{align} \label{conjugated_propagator}
	\widetilde{G}(t) = \mathcal{T}_\varphi \circ G(t) \circ \mathcal{T}_\varphi^*, \ \ t \ge 0.
\end{align}

Our goal is to prove that, for all $t \ge 0$, the semigroup $\widetilde{G}(t)$ is a metaplectic Fourier integral operator in the sense of Section 4 whose underlying complex canonical transformation is the Hamilton flow of $-i\tilde{q}$ at time $t$. To this end, we recall from  \cite{GaborSingularities} that $\kappa_t:=\exp{(tH_{-iq})}: \C^{2n} \rightarrow \C^{2n}$, the Hamilton flow of $-iq$, is positive relative to $\R^{2n}$ for each $t \ge 0$. To see this, write
\begin{align*}
    q(X) = Q X \cdot X, \ \ \ X\in \C^{2n},
\end{align*}
where $Q \in M_{2n \times 2n}(\C)$ is symmetric. Let
\begin{align} \label{definition Hamilton matrix of q}
    F = J Q
\end{align}
be the Hamilton matrix of $q$. We can express $\kappa_t$ in terms of $F$ as follows:
\begin{align} \label{Hamilton flow of -iq in terms of F}
    \kappa_t = e^{-2itF}, \ \ \ t \in \R.
\end{align}
Because the unique antilinear involution of $\C^{2n}$ fixing $\R^{2n}$ is the usual map of complex conjugation $X \mapsto \overline{X}$,
the complex canonical transformation $\kappa_t$ is positive relative to $\R^{2n}$ for all $t \ge 0$ if and only if for every $X \in \C^{2n}$ the real-valued function
\begin{align} \label{definition little g}
    r(t) = \frac{1}{i}\left(\sigma(\kappa_t(X), \overline{\kappa_t(X)})-\sigma(X,\overline{X})\right), \ \ \ t \in \R,
\end{align}
is non-negative for all $t \ge 0$. Recalling that
\begin{align}
    \sigma(X,Y) = JX \cdot Y, \ \ \ X, Y \in \C^n,
\end{align}
we see that (\ref{definition little g}) may be rewritten as
\begin{align} \label{definition little g rewritten}
    r(t) = \frac{1}{i} \left(J e^{-2itF} X \cdot e^{2it \overline{F}} \overline{X} - J X \cdot \overline{X} \right), \ \ \ t \ge 0.
\end{align}
Differentiating (\ref{definition little g rewritten}) with respect to $t$ gives
\begin{align} \label{g primed of t}
    r'(t) = 2(\overline{F}^T J - J F) e^{-2itF} X \cdot \overline{e^{-2itF} X}.
\end{align}
In view of (\ref{definition Hamilton matrix of q}), we have
\begin{align} \label{non-negative of matrices}
    \overline{F}^T J - JF = 2 \textrm{Re} \ Q \ge 0.
\end{align}
Integrating (\ref{g primed of t}) from $0$ to $t$ and using (\ref{non-negative of matrices}), we find that (\ref{definition little g}) is non-negative for $t \ge 0$.

Let $\tilde{\kappa}_t := \exp{(tH_{-i\tilde{q}})}$, $t \in \R$, be the Hamilton flow of $-i\tilde{q}$. By Jacobi's theorem,
\begin{align*}
    \tilde{\kappa}_t = \kappa_\varphi \circ \kappa_t \circ \kappa_\varphi^{-1}, \ \ \ t \in \R.
\end{align*}
Because $\kappa_t^* \sigma = \sigma$ and $\kappa_\varphi(\R^{2n}) = \Lambda_\Phi$, the flow $\tilde{\kappa}_t$ is positive relative to $\Lambda_\Phi$ for each $t \ge 0$. From the results of \cite{ComplexFIOs} (see also the discussion at the beginning of Section 5 above) there is a one-parameter family $\Phi_t$, $t \ge 0$, of strictly plurisubharmonic quadratic forms on $\C^n$ with $\Phi_t \le \Phi$ for all $t \ge 0$ such that $\Phi_0 = \Phi$ and 
\begin{align} \label{image of I-Lagrangian}
    \tilde{\kappa}_t(\Lambda_\Phi) = \Lambda_{\Phi_t}, \ \ \ t \ge 0.
\end{align}
It turns out that $\Phi_t$, $t \ge 0$, satisfies a natural eikonal equation associated to $\tilde{\kappa}_t$. To the function $\Phi(t,z) = \Phi_t(z)$, we may associate the submanifold
\begin{align*}
    \set{(t,\tau; z, \zeta)}{t \ge 0, z \in \C^n, \tau = \frac{\p \Phi}{\p t}, \ \zeta = \frac{2}{i} \frac{\p \Phi}{\p z}}
\end{align*}
of $\R^2_{t, \tau} \times \C^{2n}_{z,\zeta}$, which is Lagrangian with respect to the real symplectic form
\begin{align} \label{big real symplectic form}
    d\tau \wedge dt - \textrm{Im} \ \sigma.
\end{align}
For $g \in \textrm{Hol}(\C^{2n})$, we denote by $\widehat{H_g}$ the real vector field on $\C^{2n}$ corresponding to the holomorphic vector field $H_g$:
\begin{align*}
    \widehat{H_g} = H_g + \overline{H_g}.
\end{align*}
From the general relation \cite{AnalyticMicrolocal_Analysis}, we know that
\begin{align*}
    \widehat{H_{-i\tilde{q}}} = H^{-\textrm{Im} \ \sigma}_{\textrm{Re} \ \tilde{q}},
\end{align*}
where $H^{-\textrm{Im} \ \sigma}_{\textrm{Re} \ \tilde{q}}$ denotes the Hamilton vector field of $\textrm{Re} \ \tilde{q}$ on $\C^{2n}$ taken with respect to $-\textrm{Im} \ \sigma$. Applying the Hamilton-Jacobi theory (see Chapter 1 of \cite{dimassi_sjostrand}) with respect to the real symplectic form (\ref{big real symplectic form}), we find that $\Phi(t,z)$ satisfies the eikonal equation
\begin{align} \label{Eikonal equation for Phi}
\begin{cases}
\frac{\p \Phi}{\p t}(t,z) + \textrm{Re} \ \tilde{q} \left(z, \frac{2}{i} \frac{\p \Phi}{\p z}(t,z) \right) = 0, \ \ t \ge 0, \ z \in \C^n, \\
\Phi(0,\cdot) = \Phi \  \textrm{on $\C^n$.}
\end{cases}
\end{align}



Now we prove that $\widetilde{G}(t)$, $t \ge 0$, is a metaplectic Fourier operator in the complex domain whose underlying complex canonical transformation at time $t$ is $\tilde{\kappa}_t$. For $0 \le t \ll 1$, this may be accomplished by a standard geometrical optics construction (see for instance Section 3 of \cite{QuadraticOperators} or Section 2 of \cite{SubellipticEstimates}). However, it is actually possible to construct $\widetilde{G}(t)$ as a metaplectic Fourier integral operator directly in the Bergman form (\ref{definition of G}) for all $t \ge 0$. To the best of our knowledge, the idea of representing evolution semigroups on the FBI transform side as Fourier integral operators in Bergman form was introduced by J. Sj\"{o}strand in the work \cite{semigroup2resolvent}. The technique we present below may be viewed as a linearized version of the construction given in \cite{semigroup2resolvent}, valid for all positive times thanks to the positivity of $\tilde{\kappa}_t$ relative to $\Lambda_\Phi$.

We search for $\widetilde{G}(t)$ of the form
\begin{align} \label{Bergman ansatz}
	\widetilde{G}(t)u(z) = \hat{a}(t) \int_{\C^n} e^{2 \Psi_t(z,\overline{w})} u(w) e^{-2\Phi(w)} \, L(dw), \ \ u \in H_\Phi(\C^n),
\end{align}
where $\hat{a} \in C^\infty([0,\infty); \C)$ is non-vanishing and $\Psi_t(\cdot, \cdot)$ is a holomorphic quadratic form on $\C^{2n}$ with coefficients depending smoothly on $t$ for $t \ge 0$. Our objective is to choose $\hat{a}$ and $\Psi_t(\cdot, \cdot)$ so that $\widetilde{G}(t)$ solves the operator initial value problem
\begin{align} \label{operator_initial_value_problem}
	\begin{cases}
		\p_t \widetilde{G}(t) + \tilde{q}^w(z,D) \widetilde{G}(t) = 0, \ \ t \ge 0,\\
		\widetilde{G}(0) = I \ \textrm{on} \ H_\Phi(\C^n).
	\end{cases}
\end{align}
To this end, let us rewrite (\ref{Bergman ansatz}) in the form
\begin{align} \label{Bergman_ansatz_contour_integral}
	\widetilde{G}(t)u = a(t) \int_{\Gamma} e^{2 \Psi_t(z,\theta) - 2 \Psi(w,\theta)} u(w) \, dw \wedge d\theta, \ \ u \in H_\Phi(\C^n),
\end{align}
where
\begin{align} \label{anti-diagonal_contour}
	\Gamma = \set{(w,\theta) \in \C^{2n}}{\theta = \overline{w}}
\end{align}
is the anti-diagonal in $\C^{2n}$, $a(t) = (i/2)^{n} \hat{a}(t)$, $t \ge 0$, and $\Psi(\cdot, \cdot)$ is the polarization of $\Phi$. Thanks to (\ref{quadratic_differential_operator_on_the_FBI_side_formula}) and the well-known formula for the conjugation of a Weyl differential operator by a quadratic exponential (see, for instance, the proof of Theorem 10.6 in \cite{SemiclassicalAnalysis}), we know that
\begin{align} \label{conjugated_operator}
	e^{-2\Psi_t(z,\theta)} \circ (\p_t + \tilde{q}^w(z,D)) \circ e^{2 \Psi_t(z, \theta)} = \p_t+2 \p_t \Psi_t(z,\theta)+\tilde{q}^w_{\frac{2}{i} \Psi_t(\cdot, \theta)}(z,D), \ \ \theta \in \C^n,
\end{align}
where
\begin{align}
	\tilde{q}_{\frac{2}{i}\Psi_t(\cdot, \theta)}(z,\zeta) = \tilde{q} \left(z, \zeta+\frac{2}{i}\p_z \Psi_t(z,\theta)\right), \ \ (z,\zeta) \in \C^{2n}, \ \theta \in \C^n.
\end{align}
Assume that $\Psi_t(\cdot, \theta)$ satisfies the eikonal equation
\begin{align}
	2 \p_t \Psi_t(z,\theta) + \tilde{q}\left(z, \frac{2}{i} \p_z \Psi_t(z,\theta)\right) = 0, \ \ t \ge 0, \ z \in \C^n, \ \theta \in \C^n.
\end{align}
Using that $\tilde{q}$ is quadratic, we see that the conjugated operator (\ref{conjugated_operator}) is equal to
\begin{align}
	\p_t + \left(\p_\zeta \tilde{q} \left(z, \frac{2}{i} \p_z \Psi_t(z,\theta)\right) \cdot \zeta \right)^w + \frac{1}{2} \left(\p^2_{\zeta \zeta} \tilde{q}\right)D_z \cdot D_z, \ \ z \in \C^n, \ \theta \in \C^n, \ t \ge 0.
\end{align}
Introducing the holomorphic vector field
\begin{align}
	\nu(z, \p_z) = \p_\zeta \tilde{q}\left(z, \frac{2}{i} \p_z \Psi_t(z,\theta) \right) \cdot \p_z, \ \ z \in \C^n,
\end{align}
we deduce
\begin{align}
	e^{-2\Psi_t(z,\theta)} \circ (\p_t + \tilde{q}^w(z,D)) \circ e^{2 \Psi_t(\cdot, \theta)} = \p_t + \frac{1}{i} \nu(z,\p_z)+\frac{1}{2i} \textrm{div}(\nu) + \frac{1}{2} \left(\p^2_{\zeta \zeta} \tilde{q} \right) D_z \cdot D_z
\end{align}
for $z\in \C^n$, $\theta \in \C^n$, and $t \ge 0$. Here $\textrm{div}(\nu)$ denotes the holomorphic divergence of $\nu$,
\begin{align}
	\textrm{div}(\nu) = \sum_{j=1}^n \p^2_{\zeta_j z_j} \tilde{q} + \sum_{j=1}^n \sum_{k=1}^n \left(\p^2_{\zeta_j \zeta_k} \tilde{q}\right) \left(\frac{2}{i} \p^2_{z_j z_k} \Psi_t \right) = \trace{\p^2_{\zeta z} \tilde{q} + \p^2_{\zeta \zeta} \tilde{q} \cdot \frac{2}{i} \p^2_{zz} \Psi_t}.
\end{align}
Thus, if we are to have
\begin{align}
	\left(\p_t+\tilde{q}^w(z,D) \right)\left(e^{2 \Psi_t(\cdot, \theta)} a(t) \right) \equiv 0 \ \textrm{on $\C^n$}, \ \ \theta \in \C^n, \ \ t \ge 0,
\end{align}
it suffices to choose $a(t)$ so that
\begin{align}
	a'(t) + \frac{1}{2i} \beta(t) a(t) = 0, \ \ t \ge 0,
\end{align}
where
\begin{align}
	\beta(t) = \trace{\p^2_{\zeta z} \tilde{q} + \p^2_{\zeta \zeta} \tilde{q} \cdot \frac{2}{i} \p^2_{zz} \Psi_t}, \ \ t \ge 0.
\end{align}
Demanding also that $\Psi_t(z,\theta)|_{t=0} = \Psi(z,\theta)$ and $a(0) = (i/2)^n C_\Phi$, where $C_\Phi$ is as in (\ref{Bergman_constant}), we may ensure that $\widetilde{G}(0)$ coincides with the Bergman projection (\ref{definition_Bergman_projector}) and hence that the initial condition $\widetilde{G}(0) = I$ on $H_\Phi(\C^n)$ is satisfied. We conclude that if we are to produce a solution $\widetilde{G}(t)$ of the operator initial value problem (\ref{operator_initial_value_problem}) of the form (\ref{Bergman ansatz}), we should choose $\Psi_t(\cdot, \cdot)$ so that
\begin{align} \label{initial_value_problem_for_phase}
	\begin{cases}
		2\p_t \Psi_t(z, \theta) + \tilde{q} \left(z, \frac{2}{i} \p_z \Psi_t(z,\theta) \right) = 0, \ \ z, \theta \in \C^n, \ t \ge 0, \\
		\Psi_0(z,\theta) = \Psi(z,\theta), \ \ z \in \C^n, \ \theta \in \C^n,
	\end{cases}
\end{align}
and choose $\hat{a}(t)$ so that
\begin{align} \label{IVP_for_a}
	\begin{cases}
		\hat{a}'(t) + \frac{1}{2i}\beta(t) \hat{a}(t) = 0, \ \ t\ge 0, \\
		\hat{a}(0) = C_\Phi.
	\end{cases}
\end{align}
As the initial value problem (\ref{IVP_for_a}) can be solved by elementary methods once $\Psi_t(\cdot, \cdot)$ is known, we will focus our attention on solving the problem (\ref{initial_value_problem_for_phase}). We note that, since (\ref{IVP_for_a}) is a linear ordinary differential equation and $C_\Phi \neq 0$, the solution $\hat{a}(t)$ of (\ref{IVP_for_a}) will be non-vanishing for all $t \ge 0$. Taking the real part of (\ref{initial_value_problem_for_phase}) gives
\begin{align} \label{after_taking_real_part}
	\begin{cases}
		\p_t \left[2 \textrm{Re} \ \Psi_t(z, \theta) \right] + \textrm{Re} \ \tilde{q} \left(z, \frac{2}{i} \p_z \left[2 \textrm{Re} \ \Psi_t(z,\theta)\right]\right) = 0, \ \ z, \theta \in \C^n, \ t\ge 0, \\
		2 \textrm{Re} \ \Psi_0(z,\theta) = 2 \textrm{Re} \ \Psi(z,\theta), \ \ z \in \C^n, \ \theta \in \C^n.
	\end{cases}
\end{align}
We observe that (\ref{after_taking_real_part}) implies 
\begin{align} \label{mapping_of_eikonals}
	\Lambda_{2 \textrm{Re} \ \Psi_t(\cdot, \theta)} = \tilde{\kappa}_t \left(\Lambda_{2 \textrm{Re} \ \Psi(\cdot, \theta)}\right), \ \ \theta \in \C^n, \ t \ge 0,
\end{align}
where $\Lambda_{2 \textrm{Re} \ \Psi_t(\cdot, \theta)}$ and $\Lambda_{2 \textrm{Re} \ \Psi(\cdot, \theta)}$ denote the $\C$-Lagrangian subspaces of $\C^{2n}$ given by
\begin{align}
	\Lambda_{2 \textrm{Re} \ \Psi_t(\cdot, \theta)} = \set{\left(z, \frac{2}{i} \p_z \left[2 \textrm{Re} \ \Psi_t(z, \theta) \right]\right)}{z \in \C^n}, \ \ \theta \in \C^n, \ \ t \ge 0,
\end{align}
and
\begin{align} \label{evolution_of_eikonals}
	\Lambda_{2 \textrm{Re} \ \Psi(\cdot, \theta)} = \set{\left(z, \frac{2}{i} \p_z \left[2 \textrm{Re} \ \Psi(z, \theta) \right]\right)}{z \in \C^n}, \ \ \theta \in \C^n,
\end{align}
respectively. As a consequence of the fundamental estimate (\ref{fundamental_estimate}), there is a constant $c>0$ such that
\begin{align}
	2 \textrm{Re} \ \Psi(z,0) \le \Phi(z) - c \abs{z}^2, \ \ z \in \C^n.
\end{align}
From  Theorem 2.1 of \cite{ComplexFIOs} it follows that the $\C$-Lagrangian subspace $\Lambda_{2 \textrm{Re} \ \Psi(\cdot, 0)}$ of $\C^{2n}$ is positive relative to $\Lambda_{\Phi}$. As $\tilde{\kappa}_t$ is positive relative to $\Lambda_\Phi$ for all $t \ge 0$, (\ref{mapping_of_eikonals}) implies that $\Lambda_{2 \textrm{Re} \ \Psi_t(\cdot, 0)}$ is positive relative to $\Lambda_\Phi$ for all $t \ge 0$. This observation, combined with Theorem 2.1 of \cite{ComplexFIOs}, implies that (\ref{after_taking_real_part}) may be solved for all $t \ge 0$, first in the case $\theta = 0$, and then for general $\theta \in \C^n$. Thus we obtain a holomorphic quadratic form $\Psi_t(\cdot, \cdot)$ depending analytically on $t$ for $t \ge 0$ that solves the initial value problem (\ref{initial_value_problem_for_phase}). It follows that $\widetilde{G}(t)$ given by (\ref{Bergman ansatz}) satisfies (\ref{operator_initial_value_problem}).

Finally, let us check that $\widetilde{G}(t)$ is a metaplectic Fourier integral operator in the sense of Section 4. Writing $\widetilde{G}(t)$ as the contour integral (\ref{Bergman_ansatz_contour_integral}), we see that $\widetilde{G}(t)$ is of the form (\ref{Formal FIO}) with phase function
\begin{align} \label{Bergman_phase_function}
	\phi_t(z,w,\theta) =  \frac{2}{i} \Psi_t(z,\theta) - \frac{2}{i} \Psi(w,\theta), \ \ (z,w,\theta) \in \C^{3n}, \ \ t \ge 0.
\end{align}
Since $\Phi$ is strictly plurisubharmonic, the phase $\phi_t$ is easily seen to satisfy H\"{o}rmander's non-degeneracy condition (\ref{nondegeneracy_condition}) for every $t \ge 0$. Moreover, the relation (\ref{mapping_of_eikonals}) implies that the phase $\phi_t$ generates $\textrm{graph}(\tilde{\kappa}_t)$ in the sense of (\ref{canonical transformation associated to metaplectic FIO}) for all $t \ge 0$. As $\p_\theta \phi_t(z,w,\theta) = 0$ for $(z,w,\theta) \in \C^{3n}$ and $t \ge 0$ if and only if $\p_\theta \Psi_t(z,\theta) = \p_\theta \Psi(w,\theta)$, it follows that $\tilde{\kappa}_t: \C^{2n} \rightarrow \C^{2n}$ is given implicitly by
\begin{align}
	\tilde{\kappa}_t: \left(w, \frac{2}{i} \p_w \Psi(w,\theta) \right) \mapsto \left(z, \frac{2}{i} \p_z \Psi_t(z,\theta) \right), \ \p_\theta \Psi_t(z,\theta) = \p_\theta \Psi(w,\theta), \ \ z,w, \theta \in \C^n, \ t \ge 0.
\end{align}
We conclude that for all $t \ge 0$ the operator $\widetilde{G}(t)$ is indeed a metaplectic Fourier integral operator in the complex domain with underlying canonical transformation $\tilde{\kappa}_t$. By Proposition \ref{main imported proposition}, (\ref{Bergman ansatz}) is the Bergman form of $\widetilde{G}(t)$ for every $t \ge 0$.

Let us verify that
\begin{align} \label{identity_level_of_distributions}
	\forall u \in \mathcal{S}'(\R^n), \ \forall t\ge0: \ \ \mathcal{T}_\varphi G(t) u = \widetilde{G}(t) \mathcal{T}_\varphi u.
\end{align}
From the work \cite{GeneralizedMehler}, we know that for every $t \ge 0$ the operator $G(t): L^2(\R^n) \rightarrow L^2(\R^n)$ extends uniquely to a sequentially continuous linear transformation $G(t): \mathcal{S}'(\R^n) \rightarrow \mathcal{S}'(\R^n)$. Thus, by Propositions \ref{density_proposition}, \ref{dual_space_proposition} and \ref{mapping properties of metaplectic FIO}, it suffices to show that for every $u \in \mathcal{S}'(\R^n)$ and $t \ge 0$, we have
\begin{align} \label{verification_by_duality}
	\int_{\C^n} \mathcal{T}_\varphi G(t) u(z) \overline{v(z)} e^{-2\Phi(z)} \, L(dz) = \int_{\C^n} \widetilde{G}(t) \mathcal{T}_\varphi u(z) \overline{v(z)} e^{-2\Phi(z)} \, L(dz)
\end{align}
for all $v \in H^\infty_\Phi(\C^n)$. For any $u \in \mathcal{S}'(\R^n)$, $v \in H^{\infty}_\Phi(\C^n)$, and $t \ge 0$, we have
\begin{align}
	\begin{split}
		\int_{\C^n} \mathcal{T}_\varphi G(t) u(z) \overline{v(z)} e^{-2\Phi(z)} \, L(dz) &= \int_{\C^n} \langle G(t)u, c_\varphi e^{i \varphi(z, \cdot)} \overline{v(z)} e^{-2\Phi(z)} \rangle \, L(dz) \\
		&=\langle G(t)u, \overline{\mathcal{T}_\varphi^* v} \rangle \\
		&= \langle u, \overline{G(t)^* \mathcal{T}^*_\varphi v} \rangle,
	\end{split}
\end{align}
where $\mathcal{T}_\varphi^*$ is the adjoint of $\mathcal{T}_\varphi: L^2(\R^n) \rightarrow H_\Phi(\C^n)$ and $G(t)^*$ is the adjoint of $G(t)$ taken in the sense of distributions. The identity (\ref{conjugated_propagator}) implies
\begin{align}
	G(t)^* \mathcal{T}^*_\varphi v = (\mathcal{T}_\varphi G(t))^* v = (\widetilde{G}(t) \mathcal{T}_\varphi)^* v = \mathcal{T}^*_\varphi \widetilde{G}(t)^*v.
\end{align}
Thus
\begin{align} \label{intermediate_equality}
	\langle u, \overline{G(t)^* \mathcal{T}^*_\varphi v} \rangle = \langle u, \overline{\mathcal{T}^*_\varphi \widetilde{G}(t)^*v} \rangle = \int_{\C^n} \mathcal{T}_\varphi u(w) \overline{\widetilde{G}(t)^* v(w)} e^{-2\Phi(w)} \, L(dw).
\end{align}
As a consequence of (\ref{Bergman ansatz}), 
\begin{align} \label{adjoint_of_propagator_FBI_side}
	\widetilde{G}(t)^* v(w) = \overline{\hat{a}(t)} \int_{\C^n} e^{2 \overline{\Psi_t(z,\overline{w})}} v(z) e^{-2\Phi(z)} \, L(dz).
\end{align}
Putting (\ref{adjoint_of_propagator_FBI_side}) into (\ref{intermediate_equality}), interchanging the order of integration, and using (\ref{Bergman ansatz}) gives (\ref{verification_by_duality}).

Having established that $\widetilde{G}(t)$ is a metaplectic Fourier integral operator whose underlying complex canonical transformation at time $t$ is $\tilde{\kappa}_t$, we can apply the results of Section 5 to study the propagation of $1/2$-Gelfand-Shilov singularities by the semigroup $G(t)$, via the identity (\ref{identity_level_of_distributions}). We begin by giving a characterization of the singular space $S$ of $q$ in terms of the intersection $\Lambda_\Phi \cap \Lambda_{\Phi_t}$ for $t>0$.
\begin{proposition} \label{geometric_proposition} For all $t>0$,
	\begin{align} \label{characterization of singular space claim}
    S = \kappa_\varphi^{-1}(\Lambda_\Phi \cap \Lambda_{\Phi_t}).
\end{align}
\end{proposition}
\begin{proof}
By Theorem 1.1 of \cite{ComplexFIOs}, the positivity of $\tilde{\kappa_t}$ relative to $\Lambda_\Phi$ implies that the quadratic form $\Phi- \Phi_t$ is non-negative for all $0 \le t < \infty$. Thus,
\begin{align} \label{first equivalence}
    \forall z \in \C^n, \ \forall  \ 0\le t < \infty: \ \left(z, \frac{2}{i} \frac{\p \Phi}{\p z}(z) \right) \in \Lambda_{\Phi_t} &\iff \frac{\p \Phi}{\p z}(z) = \frac{\p \Phi_t}{\p z}(z) \iff \Phi(z)-\Phi_t(z)=0.
\end{align}
According to (\ref{Eikonal equation for Phi}),
\begin{align} \label{want to show left hand side is non positive}
    \frac{\p \Phi_t}{\p t}(z) = -\textrm{Re} \ \tilde{q}\left(z, \frac{2}{i} \frac{\p \Phi_t}{\p z}(z) \right), \ z \in \C^n, \ 0\le t < \infty.
\end{align}
For any fixed $z \in \C^n$ and $0 \le t < \infty$, the point $\left(z, \frac{2}{i} \frac{\p \Phi_t}{\p z}(z) \right)$ belongs to $\Lambda_{\Phi_t} = \tilde{\kappa_t}(\Lambda_{\Phi})$. It follows that for any $z \in \C^n$ there is an $X \in \Lambda_{\Phi}$ such that
\begin{align*}
    \left(z, \frac{2}{i} \frac{\p \Phi_t}{\p z}(z) \right) = \tilde{\kappa_t}(X).
\end{align*}
Since $\tilde{q}$ is invariant under the flow $\tilde{\kappa_t}$,
\begin{align} \label{Hamilton flow preserves positivity}
    \textrm{Re} \ \tilde{q} \left(z, \frac{2}{i} \frac{\p \Phi_t}{\p z}(z) \right) = \textrm{Re} \ \tilde{q}(X).
\end{align}
As $\tilde{q}|_{\Lambda_{\Phi}} \ge 0$, (\ref{Hamilton flow preserves positivity}) and (\ref{want to show left hand side is non positive}) together imply that
\begin{align}
    \frac{\p \Phi_t}{\p t}(z) \le 0
\end{align}
for every $z \in \C^n$ and $0 \le t < \infty$. Hence,
\begin{align*}
    \forall z \in \C^n, \ \forall \ 0\le t < \infty: \ \left(z, \frac{2}{i} \frac{\p \Phi}{\p z}(z) \right) \in \Lambda_{\Phi_t} \iff \Phi(z)-\Phi_s(z) = 0 \ \textrm{for all} \ 0 \le s \le t.
\end{align*}
Because the quadratic form $\Phi-\Phi_s$ is non-negative, $\Phi(z)-\Phi_s(z) = 0$ if and only if $\p_z \Phi(z) = \p_z \Phi_s(z)$, and we have
\begin{align*}
    \forall z \in \C^n, \ \forall \ 0\le t < \infty: \ \left(z, \frac{2}{i} \frac{\p \Phi}{\p z}(z) \right) \in \Lambda_{\Phi_t} \iff \left(z, \frac{2}{i} \frac{\p \Phi}{\p z}(z) \right) \in \Lambda_{\Phi_s} \ \textrm{for all} \ 0 \le s \le t.
\end{align*}
Therefore
\begin{align} \label{intersection of Lambda phi's}
    \Lambda_{\Phi} \cap \Lambda_{\Phi_t} = \bigcap_{0\le s \le t} \Lambda_{\Phi} \cap \Lambda_{\Phi_s}
\end{align}
for every $t > 0$. Applying $\kappa_{\varphi}^{-1}$ to both sides of $(\ref{intersection of Lambda phi's})$ and using Jacobi's theorem and (\ref{Hamilton flow of -iq in terms of F}), we find that
\begin{align*}
    \kappa_{\varphi}^{-1}(\Lambda_{\Phi} \cap \Lambda_{\Phi_t}) = \set{X \in \R^{2n}}{e^{2isF} X \in \R^{2n} \ \textrm{for all} \ 0 \le s \le t }
\end{align*}
for all $t>0$. Thus, if $X \in \C^{2n}$, then
\begin{align} \label{first condition on imaginary part}
    X \in \kappa_{\varphi}^{-1}(\Lambda_\Phi \cap \Lambda_{\Phi_t})
   \iff \textrm{Im}\left(e^{2isF} X \right) = 0
\end{align}
for every $0\le s \le t$. From the discussion on page 22 of \cite{GaborSingularities} and the real analyticity of the mapping $t \mapsto \left(\textrm{Im} \ e^{2itF} \right)(X)$ for any fixed $X \in \R^{2n}$, we know that
\begin{align} \label{dynamical interpretation of singular space}
    S = \bigcap_{0\le s \le t} \ker{\left(\textrm{Im} \ e^{2isF}\right)} \cap \R^{2n} = \bigcap_{s \in \R} \ker{(\textrm{Im} \ e^{2isF})} \cap \R^{2n}
\end{align}
for every $t>0$. Therefore (\ref{characterization of singular space claim}) must hold for every $t>0$.
\end{proof}

We now conclude the proof of Theorem \ref{main theorem of the paper}. First, we verify that $S$ is invariant under $\kappa_t$ for every $t \in \R$. From (\ref{Hamilton flow of -iq in terms of F}) and (\ref{definition of singular space}), we see that
\begin{align} \label{acts_as_imaginary_flow}
	\kappa_t(X) = e^{tH_{\textrm{Im} \ q}} X = \sum_{k=0}^\infty \frac{(2t)^k}{k!} (\textrm{Im} \ F)^k X
\end{align}
for all $X \in S$ and $t \in \R$. Now, let $X \in S$, $Y \in \C^{2n}$, and $t \in \R$ be such that
\begin{align} \label{showing singular space is invariant by Hamilton flow}
    \kappa_t(X) = Y.
\end{align}
As a consequence of (\ref{dynamical interpretation of singular space}), we have $Y \in \R^{2n}$. Also, (\ref{acts_as_imaginary_flow}) and (\ref{definition of singular space}) together imply
\begin{align*}
(\textrm{Re} \ F)(\textrm{Im} \ F)^j Y = \sum_{k=0}^\infty \frac{(2t)^k}{k!} (\textrm{Re} \ F)(\textrm{Im} \ F)^{j+k} X = 0, \ \ j \in \N.
\end{align*}
Thus
\begin{align} \label{first_Hamilton_inclusion}
    \kappa_t(S) \subset S, \ \ \ t \in \R.
\end{align}
Since (\ref{first_Hamilton_inclusion}) implies that
\begin{align}
	\kappa_{-t}(S) \subset S, \ \ t \in \R,
\end{align}
holds as well, we deduce that
\begin{align} \label{singular space is invariant under Hamilton flow}
	\kappa_t(S) = S, \ \ t \in \R.
\end{align}

The invariance of $S$ under $\kappa_t$ for every $t \in \R$, Jacobi's theorem, and Proposition \ref{geometric_proposition} give
\begin{align}
    \tilde{\kappa}_t(\Lambda_\Phi \cap \Lambda_{\Phi_t}) = \Lambda_\Phi \cap \Lambda_{\Phi_t}
\end{align}
for every $t \ge 0$. Thus, the metaplectic Fourier integral operator $\widetilde{G}(t)$ and its underlying canonical transformation $\tilde{\kappa}_t$ satisfy the hypotheses of Theorem \ref{main theorem} for every $t \ge 0$. Let $\textrm{pr}_\Phi = \pi_1|_{\Lambda_\Phi}$, $\textrm{pr}_{\Phi_t} = \pi_1|_{\Lambda_{\Phi_t}}$ for $t \ge 0$, and $\tilde{\kappa}^\flat_t = \textrm{pr}_{\Phi_t} \circ \tilde{\kappa}_t \circ \textrm{pr}_{\Phi}^{-1}$ for $t \ge 0$. By Theorem \ref{main theorem}, we have
\begin{align} \label{almost there}
    \textrm{WF}^{1/2}_\Phi(\widetilde{G}(t)u) = \tilde{\kappa}^\flat_t\left(\textrm{WF}^{1/2}_\Phi(u) \right) \cap \textrm{Rad}(\Phi-\Phi_t)
\end{align}
for every $u \in H^{-\infty}_\Phi(\C^n)$ and every $t \ge 0$. Applying $(\textrm{pr}_\Phi \circ \kappa_\varphi|_{\R^{2n}})^{-1}$ to both sides of (\ref{almost there}) and using Jacobi's theorem, ({\ref{identity_level_of_distributions}), (\ref{characterization of singular space claim}), (\ref{singular space is invariant under Hamilton flow}), and Proposition \ref{image_of_FBI_transform}, we get
\begin{align} \label{almost_almost_there}
    \textrm{WF}^{1/2}(G(t)u_0) = \kappa_t(\textrm{WF}^{1/2}(u_0) \cap S)
\end{align}
for every $u_0 \in \mathcal{S}'(\R^n)$ and every $t > 0$. From (\ref{almost_almost_there}) and (\ref{acts_as_imaginary_flow}), we may therefore deduce
\begin{align}
    \textrm{WF}^{1/2}(G(t)u_0) = \exp{(tH_{\textrm{Im} \ q})}(\textrm{WF}^{1/2}(u_0) \cap S)
\end{align}
for every $u_0\in \mathcal{S}'(\R^n)$ and $t > 0$. The proof of Theorem \ref{main theorem of the paper} is complete.

\bibliographystyle{plain}
\bibliography{references}

\end{document}